\documentclass[
draft=false,DIV=classic,fontsize=10pt]{scrartcl}
\KOMAoptions{DIV=15}

\usepackage[latin1]{inputenc}

\usepackage[ngerman,english]{babel}
\usepackage{amsmath,amssymb,amsthm}
\usepackage{
           ,enumerate
           }
\usepackage[usenames,dvipsnames]{color}
\usepackage[left=22mm,top=3cm,right=25mm,bottom=3cm]{geometry}

\newtheorem{theorem}{Theorem}[section]
\newtheorem{definition}[theorem]{Definition}
\newtheorem{proposition}[theorem]{Proposition}
\newtheorem{corollary}[theorem]{Corollary}
\newtheorem{lemma}[theorem]{Lemma}
\newtheorem{remark}[theorem]{Remark}
\newtheorem{example}[theorem]{Example}
\newtheorem{examples}[theorem]{Examples}
\newtheorem{foo}[theorem]{Remarks}
\newenvironment{Example}{\begin{example}\rm}{\end{example}}

\newenvironment{Remark}{\begin{remark}\rm}{\end{remark}}

\newcommand{\be}{\begin{equation}}
\newcommand{\ee}{\end{equation}}
\newcommand{\bea}{\begin{eqnarray}}
\newcommand{\eea}{\end{eqnarray}}
\newcommand{\beas}{\begin{eqnarray*}}
\newcommand{\eeas}{\end{eqnarray*}}

\newcommand{\brak}[1]{\ensuremath{\left( #1 \right)}}
\newcommand{\crl}[1]{\ensuremath{ \left\{ #1 \right\} }}
\newcommand{\edg}[1]{\ensuremath{ \left[ #1 \right] }}
\newcommand{\ang}[1]{\ensuremath{ \left \langle #1 \right \rangle }}

\newcommand{\essinf}{\mathop{\rm ess\,inf}}
\newcommand{\esssup}{\mathop{\rm ess\,sup}}
\newcommand{\st}{{\rm st}}
\newcommand{\sst}{{\rm sst}}
\newcommand{\conv}{{\rm conv}}
\newcommand{\aff}{{\rm aff}}
\newcommand{\lin}{{\rm lin}}
\newcommand{\cone}{{\rm ccone}}

\newcommand{\inn}{{\rm int}}

\newcommand{\ri}{{\rm ri}}

\newcommand{\N}[1]{||#1||}
\newcommand{\cl}{{\rm cl}}

\newcommand{\dom}{{\rm dom \,}}

\begin{document}

\title{Conditional Analysis on $\mathbb{R}^d$
\author{\large Patrick Cheridito\footnote{Princeton University, Princeton, NJ 08544, USA.}
\and \large Michael Kupper\footnote{University of Konstanz, Universit\"atstra\ss e 10, 78464 Konstanz Germany.}
\and \large Nicolas Vogelpoth\footnote{20 Main Mill, Mumfords Mill,
19-23 Greenwich High Road, SE10 8ND London, UK.}}}

\date{\large October 2014\footnote{We thank 
Ramon van Handel, Ying Hu, Asgar Jamneshan, Mitja Stadje and 
Martin Streckfu\ss\, for fruitful discussions and helpful comments.}}

\maketitle

\begin{abstract}
\noindent 
{\bf Abstract.} This paper provides versions of classical results from 
linear algebra, real analysis and convex analysis
in a free module of finite rank over the ring $L^0$ of measurable 
functions on a $\sigma$-finite measure space. 
We study the question whether a submodule is finitely generated 
and introduce the more general concepts of $L^0$-affine sets, $L^0$-convex sets,
$L^0$-convex cones, $L^0$-hyperplanes and $L^0$-halfspaces. We investigate orthogonal complements,
orthogonal decompositions and the existence of orthonormal bases. 
We also study $L^0$-linear, $L^0$-affine,
$L^0$-convex and $L^0$-sublinear functions and introduce notions
of continuity, differentiability, directional derivatives and subgradients.
We use a conditional version of the Bolzano--Weierstrass theorem to 
show that conditional Cauchy sequences converge and 
give conditions under which conditional optimization problems have optimal solutions. 
We prove results on the separation of $L^0$-convex sets by $L^0$-hyperplanes 
and study $L^0$-convex conjugate functions. We provide a result on the existence of 
$L^0$-subgradients of $L^0$-convex functions, prove a conditional version of the Fenchel--Moreau theorem
and study conditional inf-convolutions.\\[2mm]


\end{abstract}

\setcounter{equation}{0}
\section{Introduction}
\label{sec:intr}

Let $L^0$ be the set of all real-valued measurable functions on a $\sigma$-finite measure space 
$(\Omega, {\cal F}, \mu)$, where two of them are identified if they agree $\mu$-almost everywhere. 
The purpose of this paper is to study the set $(L^0)^d$ of all 
$d$-dimensional vectors with components in $L^0$ and functions 
$f : (L^0)^d \to L^0$. Its main motivation are applications in the following two special cases:
\begin{itemize}
\item 
If $\mu$ is a probability measure, the elements of $L^0$ are random variables, and 
subsets $C \subseteq (L^0)^d$ can be understood as random sets in $\mathbb{R}^d$.
A typical function $f : (L^0)^d \to L^0$ would, for example, be a mapping that 
conditionally on ${\cal F}$, assigns to every random point $X \in (L^0)^d$ its Euclidean distance to $C$.
\item
Let $(\Omega, {\cal G}, \mu)$ be the product of a $\sigma$-finite measure space 
$(\mathbb{T}, {\cal H}, \nu)$ and a probability space $(E, {\cal E}, P)$.
If ${\cal F}$ is a sub-$\sigma$-algebra of ${\cal G}$, the elements of $L^0$ are
stochastic processes $(X_t)_{t \in \mathbb{T}}$ on $(E, {\cal E}, P)$.
A subset $C \subseteq (L^0)^d$ could, for instance, describe the set of admissible strategies in a 
stochastic control problem, and an optimal strategy could be characterized as the 
conditional optimizer of an appropriate function $f : (L^0)^d \to L^0$ over $C$.
\end{itemize}
Unless $\Omega$ is the union of finitely many atoms, 
$(L^0)^d$ is an infinite-dimensional vector space over $\mathbb{R}$. But conditioned on ${\cal F}$, 
it is only $d$-dimensional. Or put differently, it is a free module of rank $d$ over the ring $L^0$.
This allows us to derive conditional analogs of classical results from 
linear algebra, real analysis and convex analysis that depend on the fact that 
$\mathbb{R}^d$ is a finite-dimensional vector space. $L^0$-modules have been studied before; see, 
for instance, Filipovi\'c et al. (2009), Kupper and Vogelpoth (2009), Guo (2010), Guo (2011)
and the references in these papers. But since we consider free modules of finite rank, we are able to provide
stronger results under weaker assumptions,
and moreover, do not need Zorn's lemma or the axiom of choice. 
Our approach differs from standard measurable selection arguments in that we 
work modulo null-sets with respect to the measure $\mu$ 
and do not use $\omega$-wise arguments.
This has the advantage that one never leaves the world of measurable functions.
But it only works in situations where a measure $\mu$ is given, and the quantities of 
interest do not depend on $\mu$-null sets.

The results in this paper are theoretical. But they have already been applied several times:
in Cheridito and Hu (2011), they were used to describe stochastic constraints and 
characterize optimal 
strategies in a dynamic consumption and investment problem. In Cheridito and Stadje (2012) they
guaranteed the existence of a conditional subgradient.
In Cheridito et al. (2012) they were applied to show existence and uniqueness of
economic equilibria in incomplete market models.

The structure of the paper is as follows: In Section \ref{sec:algebraic} we investigate
when an $L^0$-submodule of $(L^0)^d$ is finitely generated. Then we study 
conditional orthogonality and introduce $L^0$-affine sets, $L^0$-convex sets and $L^0$-convex cones. 
It turns out that the notion of $\sigma$-stability plays a crucial role. In Section \ref{sec:limits}
we investigate almost everywhere converging sequences in $(L^0)^d$ and the corresponding notion of
closure. We define $L^0$-linear and $L^0$-affine functions $f : (L^0)^d \to (L^0)^k$ and show that 
they are continuous with respect to almost everywhere converging sequences. We also give a conditional version of the 
Bolzano--Weierstrass theorem and show that conditional Cauchy sequences converge.
Moreover, we define $L^0$-bounded sets and give a condition for $L^0$-convex sets to be $L^0$-bounded.
In Section \ref{sec:opti} we study sequentially semicontinuous and $L^0$-convex functions 
$f : (L^0)^d \to L^0$ and prove a result which guarantees that a conditional optimization problem
has an optimal solution. Section \ref{sec:open} is devoted to $L^0$-open sets, interiors 
and relative interiors. $L^0$-open sets form a topology, but they are not complements of 
sequentially closed sets. In Section \ref{sec:sep} we give strong, weak and proper 
separation results of $L^0$-convex sets by $L^0$-hyperplanes.
Section \ref{sec:convex} studies $L^0$-convex functions and introduces conditional 
notions of differentiability, directional derivatives, subgradients and convex conjugation.
We also provide results on the existence of conditional subgradients and give a conditional version 
of the Fenchel--Moreau theorem. In Section \ref{sec:convolution} we study 
conditional inf-convolutions.\\[2mm]
\noindent
{\bf Notation.} We assume $\mu(\Omega) > 0$
and define  ${\cal F}_+ := \{A\in{\cal F} : \mu[A]>0\}$.
By $L$ we denote the set of all measurable functions $X : \Omega
\to \mathbb{R} \cup \crl{\pm \infty}$, where two of them are
identified if they agree a.e. (almost everywhere). In particular,
for $X,Y \in L$, $X=Y$, $X > Y$ and $X \ge Y$ will be understood in the
a.e. sense. Analogously, for sets $A,B \in {\cal F}$, we write $A
= B$ if $\mu[A \triangle B] = 0$ and $A \subseteq B$ if $\mu[A \setminus B] = 0$. 
The set $L^0 := \crl{X \in L : |X| < \infty}$
with the a.e. order is a lattice ordered ring, and every non-empty
subset $C$ of $L$ has a least upper bound and a greatest lower bound in $L$ with 
respect to the a.e. order. We follow the usual convention in measure theory and
denote them by $\esssup C$ and $\essinf C$, respectively. It is
well-known (see for instance, Neveu, 1975) that there exist sequences $(X_n)$
and $(Y_n)$ in $C$ such that $\esssup C = \sup_n X_n$ and $\essinf C = \inf_n Y_n$. Moreover, if $C$ is
directed upwards, $(X_n)$ can be chosen such that $X_{n+1} \ge X_n$, and if $C$ is 
directed downwards, $(Y_n)$ can be chosen so that $Y_{n+1} \le Y_n$. For a set 
$A \in {\cal F}$, we denote by $1_A$ the characterisitc function of $A$, that is, 
the function $1_A : \Omega \to \crl{0,1}$ which is $1$ on $A$ and $0$ elsewhere.
If ${\cal A}$ is a subset of ${\cal F}$, we set $\esssup {\cal A} := \crl{\esssup_{A \in {\cal A}} 1_A =1} \in {\cal F}$
and $\essinf {\cal A} := \crl{\essinf_{A \in {\cal A}} 1_A =1} \in {\cal F}$.
Futhermore, we use the notation $L^0_+ := \{X \in L^0 : X \ge0\}$,  $L^0_{++} :=\{X \in L^0 : X >0\}$,  
$\overline{L} := \crl{X \in L : X > -\infty}$, $\underline{L} := \crl{X \in L : X < +\infty}$ and
${\mathbb N} := \crl{1,2, \dots}$. By $\mathbb{N}({\cal F})$ we denote the set of all measurable 
functions $N : \Omega \to \mathbb{N}$.

\setcounter{equation}{0}
\section{Algebraic structures and generating sets}
\label{sec:algebraic}

We fix $d \in{\mathbb N}$ and consider the set $(L^0)^d:=\crl{(X^1,\dots,X^d): X^i\in L^0 }$.
On $(L^0)^d$ we define the conditional inner product and conditional $2$-norm by
$$
\ang{X,Y} := \sum_{i=1}^d X^i Y^i \quad \mbox{and} \quad \N{X} := \ang{X,X}^{1/2}.
$$  
For every $A \in{\cal F}$, $1_A L^0$ is a subring of $L^0$, and provided that $\mu[A] > 0$,
$1_A (L^0)^d$ is a free $1_A L^0$-module of rank $d$ generated by the base
$1_A e_i$, $i=1, \dots, d$, where $e_i$ is the $i$-th unit vector in
$\mathbb{R}^d \subseteq (L^0)^d$. In particular, $(L^0)^d$ is a free $L^0$-module of rank $d$.

\begin{definition}
We call a subset $C$ of $(L^0)^d$
\begin{itemize}
\item stable if $1_A X + 1_{A^c} Y \in C$ for all $X,Y \in C$ and $A \in {\cal F}$;
\item $\sigma$-stable if $\sum_{n \in \mathbb{N}} 1_{A_n} X_n \in C$ for every sequence
$(X_n)_{n \in \mathbb{N}}$ in $C$ and pairwise disjoint sets $A_n \in {\cal F}$
satisfying $\Omega = \bigcup_{n \in \mathbb{N}} A_n$;
\item $L^0$-convex if $\lambda X + (1-\lambda)Y \in C$ for all $X,Y \in C$ and $\lambda \in L^0$
such that $0 \le \lambda \le 1$;
\item an $L^0$-convex cone if it is $L^0$-convex and $\lambda X \in C$ for all $X \in C$ and $\lambda
\in L^0_{++}$;
\item $L^0$-affine if $\lambda X + (1-\lambda)Y \in C$ for all $X,Y \in C$ and $\lambda \in L^0$;
\item $L^0$-linear (or an $L^0$-submodule) if $\lambda X + Y \in C$ for all $X,Y \in C$ and $\lambda \in L^0$.
\end{itemize}
For an arbitrary subset $C$ of $(L^0)^d$ and $A \in {\cal F}$, we denote
by $\st_A(C)$, $\sst_A(C)$, $\conv_A(C)$, $\cone_A(C)$, $\aff_A(C)$, $\lin_A(C)$
the smallest subset of $1_A(L^0)^d$ containing $1_A C$
that is stable, $\sigma$-stable, $L^0$-convex,
an $L^0$-convex cone, $L^0$-affine, or $L^0$-linear, respectively.
If $A = \Omega$, we just write $\st(C)$, $\sst(C)$, $\conv(C)$, $\cone(C)$,
$\aff(C)$, $\lin(C)$ for these sets.
\end{definition}

\begin{Remark}\label{Rem1}
It can easily be checked that if $C$ is a non-empty subset of $(L^0)^d$ and $A \in {\cal F}$, then
\beas
&& \st_A(C) = \crl{\sum_{n =1}^k 1_{A_n} X_n : k \in \mathbb{N}, \, X_n \in C, \, A_n \in {\cal F}, \, \bigcup_{n=1}^k A_n = A, \,
A_m \cap A_n = \emptyset \mbox{ for } m \neq n};\\
&& \sst_A(C) = \crl{\sum_{n \in \mathbb{N}} 1_{A_n} X_n :
X_n \in C, \, A_n \in {\cal F}, \, \bigcup_{n\in\mathbb{N}} A_n = A, \,
A_m \cap A_n = \emptyset \mbox{ for } m \neq n};\\
&& \conv_A(C) = \crl{\sum_{n =1}^k \lambda_n X_n : k \in \mathbb{N}, \,
X_n \in C, \, \lambda_n \in 1_A L^0_+, \, \sum_{n=1}^k \lambda_n = 1_A};\\
&& \cone_A(C) = \crl{\sum_{n =1}^k \lambda_n X_n : k \in \mathbb{N}, \,
X_n \in C, \, \lambda_n \in 1_A L^0_+, \, \sum_{n=1}^k \lambda_n \in 1_A L^0_{++}};\\
&& \aff_A(C) = \crl{\sum_{n =1}^k \lambda_n X_n : k \in \mathbb{N}, \,
X_n \in C, \, \lambda_n \in 1_A L^0, \, \sum_{n=1}^k \lambda_n =1_A};\\
&& \lin_A(C) = \crl{\sum_{n =1}^k \lambda_n X_n : k \in \mathbb{N}, \,
X_n \in C, \, \lambda_n \in 1_A L^0}.
\eeas
It follows that if $C = \crl{X_1, \dots, X_k}$ for finitely many $X_1, \dots, X_k \in (L^0)^d$, then the sets
$\conv_A(C)$, $\cone_A(C)$, $\aff_A(C)$, $\lin_A(C)$ are all $\sigma$-stable.
\end{Remark}

\begin{definition}
Let $A \in {\cal F}_+$ and $k \in \mathbb{N}$.
We call $X_1,\ldots,X_k \in (L^0)^d$ linearly independent on $A$ if
$1_A X_1, \ldots, 1_A X_k$ are linearly independent in the $1_A L^0$-module $1_A (L^0)^d$, that is,
$(0,\dots,0)$ is the only vector $(\lambda_1, \dots, \lambda_k) \in 1_A (L^0)^k$
satisfying
$$
\lambda_1 X_1 + \dots + \lambda_k X_k = 0.
$$
We say that $X_1,\ldots,X_k$ are orthogonal on $A$ if
$1_A \ang{X_i, X_j} = 0$ for $i \neq j$ and orthonormal on $A$ if in addition,
$1_A \N{X_i} = 1_A$, $1 \le i \le k$. If $X_1, \dots, X_k$ are linearly independent on 
$A$ and $\lin_A\crl{X_1, \dots, X_k} = 1_A C$ for some subset $C$ of $(L^0)^d$,
we call them a basis of $C$ on $A$. If in addition,
$X_1, \dots, X_k$ are orthogonal or orthonormal on $A$, we say $X_1, \dots, X_k$
is an orthogonal or orthonormal basis of $C$ on $A$, respectively.
\end{definition}

\begin{lemma} \label{lemmaspan}
Let $A \in {\cal F}$ and $X_1, \dots, X_k, Y \in (L^0)^d$ for some $k \in \mathbb{N}$.
Then there exists a largest subset $B \in {\cal F}$ of $A$ such that 
$1_B Y \in \lin_B \crl{X_1, \dots, X_k}$. 
\end{lemma}

\begin{proof}
The set
$$
{\cal A} := \crl{B \in {\cal F} : B \subseteq A \mbox{ and } 1_B Y \in \lin_B \crl{X_1, \dots, X_k}}
$$
is directed upwards. So it contains an increasing sequence 
$(B_n)_{n \in \mathbb{N}}$ such that $B:=\bigcup_n B_n = \esssup {\cal A}$.
$B$ is the largest element of ${\cal A}$.
\end{proof}

\begin{proposition} \label{propkl}
Let $A \in {\cal F}_+$ and $k,l \in \mathbb{N}$.
Assume $X_1, \dots, X_k \in (L^0)^d$ are linearly independent on $A$ and 
$\lin_A \crl{X_1, \dots, X_k} \subseteq \lin_A \crl{Y_1, \dots, Y_l}$
for some $Y_1, \dots, Y_l \in (L^0)^d$. Then $k \le l$. Moreover, if
$k=l$, then $Y_1, \dots, Y_l$ are linearly independent on $A$ and
$\lin_A \crl{X_1, \dots, X_k} = \lin_A \crl{Y_1, \dots, Y_l}$.
\end{proposition}

\begin{proof}
One can write $1_A X_1 = \sum_{i=1}^l \lambda_i 1_A Y_i$ for some $\lambda_i \in L^0$.
So there exists a $\sigma(1) \in \crl{1, \dots, l}$ such that 
$A_1 := A \cap \crl{\lambda_{\sigma(1)} \neq 0} \in {\cal F}_+$, and one obtains
$$
\lin_{A_1} \crl{X_1, \dots, X_k} \subseteq
\lin_{A_1} \crl{Y_1, \dots, Y_l} = \lin_{A_1}(\crl{X_1, Y_1, \dots, Y_l} 
\setminus \crl{Y_{\sigma(1)}}).
$$
In particular, if $k \ge 2$, one must have $l \ge 2$, and it follows inductively that there exist 
$A_2, \dots, A_d \in {\cal F}_+$ and an injection $\sigma : \crl{1, \dots, k} \to \crl{1, \dots, l}$ 
such that for all $i \in \crl{1, \dots, k}$,
$$
\lin_{A_i} \crl{X_1, \dots, X_k} \subseteq
\lin_{A_i} \crl{Y_1, \dots, Y_l} = 
\lin_{A_i}(\crl{X_1, \dots, X_i, Y_1, \dots, Y_l} \setminus \crl{Y_{\sigma(1)}, \dots, Y_{\sigma(i)}}).
$$
This shows that $k \le l$.

Now assume $k = l$ and $Y_1, \dots, Y_l$ are not linearly independent on $A$. Then there exist
$B \in {\cal F}_+$ and $j \in \crl{1, \dots, k}$ such that 
$$
\lin_B \crl{X_1, \dots, X_k} \subseteq \lin_B \crl{Y_1, \dots, Y_k} =
\lin_B(\crl{Y_1, \dots, Y_k} \setminus \crl{Y_j}),
$$
a contradiction to the first part of the proposition. So if $k=l$, $Y_1, \dots, Y_k$ must be 
linearly independent on $A$, and it remains to show that $\lin_A \crl{X_1, \dots, X_k} = \lin_A \crl{Y_1, \dots, Y_k}$.
To do this, we assume that $\lin_A \crl{X_1, \dots, X_k} \subsetneq \lin_A \crl{Y_1, \dots, Y_k}$.
Then $Y_j \notin \lin_A \crl{X_1, \dots, X_k}$ for at least one $j \in \crl{1, \dots, k}$. 
By Lemma \ref{lemmaspan}, there exists a largest subset $B \in {\cal F}$ of $A$ 
such that $1_B Y_j \in \lin_B \crl{X_1, \dots, X_k}$. The set $D := A \setminus B$ is in ${\cal F}_+$,
and $X_1, \dots, X_k,Y_j$ are linearly independent on $D$.
But then
$$
\lin_D \crl{X_1, \dots, X_k, Y_j} \subseteq \lin_D \crl{Y_1, \dots, Y_k},
$$
again contradicts the first part of the proposition, and the proof is complete.
\end{proof}

\begin{corollary} \label{corkl}
Let $A \in {\cal F}_+$ and $k,l \in \mathbb{N}$. Assume $X_1, \dots, X_k \in (L^0)^d$ are 
linearly independent on $A$ and 
$\lin_A \crl{X_1, \dots, X_k} = \lin_A \crl{Y_1, \dots, Y_l}$
for some $Y_1, \dots, Y_l \in (L^0)^d$ that are also linearly independent on $A$.
Then $k = l \le d$, and if $k = l = d$, one has
$\lin_A \crl{X_1, \dots, X_k} = \lin_A \crl{Y_1, \dots, Y_l} = 1_A (L^0)^d$.
\end{corollary}

\begin{proof}
The corollary follows from Proposition \ref{propkl} by noticing that 
$$
\lin_A \crl{X_1, \dots, X_k} = \lin_A \crl{Y_1, \dots, Y_l} \subseteq \lin_A(e_1, \dots, e_d) = 1_A (L^0)^d.
$$
\end{proof}

\begin{lemma} \label{lemmalargestset}
Let $C$ be a non-empty $\sigma$-stable subset of $(L^0)^d$ and 
$X_1, \dots, X_k \in (L^0)^d$ for some $k \in \mathbb{N}$. Then for given $A \in {\cal F}_+$,
each of the collections 
\be \label{A1}
\crl{B \in {\cal F}_+ : B \subseteq A \mbox{ and there exists a } Y \in C \mbox{ such that }
\N{Y} > 0 \mbox{ on } B}
\ee
and
\be \label{A2}
\crl{B \in {\cal F}_+ : B \subseteq A \mbox{ and there exists } Y \in C \mbox{ such that }
X_1, \dots, X_k,Y \mbox{ are linearly independent on } B}
\ee
is either empty or contains a largest set.
\end{lemma}

\begin{proof}
Let us denote the collection \eqref{A1} by ${\cal A}_1$ and \eqref{A2} by ${\cal A}_2$.
Both are directed upwards. So if either one of them is non-empty, it contains 
an increasing sequence of sets $B_n$ with corresponding $Y_n \in C$, $n \in \mathbb{N}$, such that 
$B:=\bigcup_n B_n = \esssup {\cal A}_i$. Since $C$ is $\sigma$-stable,
$$
Y := Y_1 1_{B_1 \cup B^c} + \sum_{n \ge 2} 1_{B_n \setminus B_{n-1}} Y_n
$$
belongs to $C$. In the first case one has $\N{Y} > 0$ on $B$, and in the second one,
$X_1, \dots, X_k,Y$ are linearly independent on $B$. This proves the lemma.
\end{proof}

\begin{theorem} \label{thmspan}
Let $C$ be a $\sigma$-stable subset of $(L^0)^d$ containing an element $X \neq 0$.
Then there exist a unique number $k \in \crl{1, \dots, d}$, unique pairwise
disjoint sets $A_0, \ldots, A_k \in{\cal F}$ and $X_1, \ldots, X_k \in C$ such that the following hold:
\begin{itemize}
\item[{\rm (i)}]
$\bigcup_{i=0}^k A_i = \Omega$ and $\mu[A_k] >0$;
\item[{\rm (ii)}]
$1_{A_0} C = \crl{0}$;
\item[{\rm (iii)}] 
For all $i \in \crl{1, \dots, k}$ satisfying $\mu[A_i] > 0$, $X_1, \dots, X_i$ is a basis of
$\lin(C)$ on $A_i$.
\end{itemize}
\end{theorem}

\begin{proof}
That $k$ and the sets $A_0, \dots, A_k$ are unique follows from Corollary \ref{corkl}. To show 
the existence of $A_i$ and $X_i$ satisfying (i)--(iii), we construct them inductively.
Since $C$ contains an element $X \neq 0$, 
it follows from Lemma \ref{lemmalargestset} that there exists a largest 
set $B_1 \in {\cal F}_+$ such that $\N{Y} > 0$ on $B_1$ for some $Y \in C$. Choose such a $Y$ and call it $X_1$.
One must have $1_{B^c_1} C = \crl{0}$.
If there exist no $B \in {\cal F}_+$ and $Y \in C$ such that $X_1,Y$ are linearly independent on $B$,
one obtains from Lemma \ref{lemmaspan} that $1_{B_1} Y \in \lin_{B_1} \crl{X_1}$ for all $Y \in C$,
and therefore, $\lin_{B_1}(C) = \lin_{B_1} \crl{X_1}$.
So one can set $k =1$, $A_0 = B^c_1$ and $A_1 = B_1$. 
On the other hand, if there exists a $B \in {\cal F}_+$ and $Y \in C$ such that $X_1,Y$ 
are linearly independent on $B$, Lemma \ref{lemmalargestset} yields a largest such set 
$B_2$ with a corresponding $X_2 \in C$. If there exists no $B \in {\cal F}_+$ and $Y \in C$ such
that $X_1, X_2, Y$ are linearly independent on $B$, then 
$\lin_{B_2}(C) = \lin_{B_2} \crl{X_1,X_2}$ and one can set $k=2$,
$A_0 = B^1_c$, $A_1 = B_1 \setminus B_2$ and $A_2 = B_2$.
Otherwise, one continues like this until there is no $B \in {\cal F}_+$ and $Y \in C$ 
such that $X_1, \dots, X_k,Y$ are linearly independent on $B$. Such a $k$ must exist and $k \le d$.
Otherwise one would have $X_1, \dots, X_{d+1} \in C$ that are linearly independent on
some $B \in {\cal F}_+$, a contradiction to Corollary \ref{corkl}. One sets
$A_0 = B^1_c$, $A_1 = B_1 \setminus B_2, \dots,$ $A_{k-1} = B_{k-1} \setminus B_k$, $A_k = B_k$.
\end{proof}

\begin{corollary} \label{coraffsigma}
Let $C$ be a non-empty $\sigma$-stable subset of $(L^0)^d$ and $A \in {\cal F}$.
Then $\aff_A(C)$ and $\lin_A(C)$ are again $\sigma$-stable.
\end{corollary}

\begin{proof}
If $1_A C = \crl{0}$, then $\aff_A(C) = \lin_A(C) = \crl{0}$, and the corollary is clear.
Otherwise, one obtains from Theorem \ref{thmspan} that there exists a $k \in \crl{1, \dots, d}$, disjoint 
sets $A_0, \ldots, A_k \in{\cal F}$ and $X_1, \ldots, X_k \in C$ such that 
$\bigcup_{i=0}^k A_i = A$, $1_{A_0} C = \crl{0}$ and for all
$i \in \crl{1, \dots, k}$ satisfying $\mu[A_i] > 0$, $X_1, \dots, X_i$ is a basis of $\lin_A(C)$ on $A_i$.
Now it can easily be verified that $\lin_A(C)$ is $\sigma$-stable. To see that $\aff_A(C)$ is 
$\sigma$-stable, one picks an $X \in 1_AC$. Then $\aff_A(C) - X = \lin_A(C-X)$ is $\sigma$-stable.
So $\aff_A(C)$ is $\sigma$-stable too.
\end{proof}

\begin{definition}
The orthogonal complement of a non-empty subset $C$ of $(L^0)^d$ is given by
$$
C^{\perp} := \crl{X \in (L^0)^d : \ang{X,Y} = 0 \mbox{ for all } Y \in C}.
$$
\end{definition}

It is clear that $C^{\perp}$ is an $L^0$-linear subset of $(L^0)^d$ satisfying 
$$
C \cap C^{\perp} \subseteq \crl{0} \quad \mbox{and} \quad C \subseteq C^{\perp \perp}.
$$
As a consequence of Theorem \ref{thmspan}, one obtains the following corollary. 

\begin{corollary} \label{corspan}
Let $C$ be a non-empty $\sigma$-stable $L^0$-linear subset of $(L^0)^d$.
Then there exist unique pairwise disjoint sets $A_0, \dots, A_d \in {\cal F}$ satisfying 
$\bigcup_{i=0}^d A_i = \Omega$ and an orthonormal 
basis $X_1, \ldots, X_d$ of $(L^0)^d$ on $\Omega$ such that 
$1_{A_0} C = \crl{0}$, $1_{A_d} C = 1_{A_d} (L^0)^d$ and
$$1_{A_i} C = \lin_{A_i} \{X_1,\ldots,X_i\}, \quad
1_{A_i} C^{\perp} = \lin_{A_i} \crl{X_{i+1}, \dots, X_d} \quad 
\mbox{for } 1 \le i \le d-1.
$$
In particular, $C + C^{\perp} = (L^0)^d$, $C \cap C^{\perp} = \crl{0}$ and $C = C^{\perp \perp}$.
\end{corollary}

\begin{proof}
The uniqueness of the sets $A_1, \dots, A_d$ follows from Corollary \ref{corkl}, and in the 
special case $C = \crl{0}$, one can choose $A_0 = \Omega$, 
$A_i = \emptyset$, $X_i = e_i$, $i=1, \dots, d$.

If $C$ is different from $\crl{0}$, it follows from Theorem \ref{thmspan} that
there exist a unique number $k \in \crl{1, \dots, d}$,
unique pairwise disjoint sets $A_0, \dots, A_k \in {\cal F}$ and
$Y_1, \ldots, Y_k \in C$ such that $\bigcup_{i=0}^k A_i = \Omega$, $\mu[A_k] > 0$,
$1_{A_0} C = \crl{0}$ and for all $i \in \crl{1, \dots, k}$
satisfying $\mu[A_i] > 0$, $Y_1, \dots, Y_i$ is a basis of $C$ on $A_i$. Let us define
$$
U_1 := 1_{A_1 \cup \dots \cup A_k} \frac{Y_1}{\N{Y_1}} \in C
$$
and
$$
Z_i := Y_i - \sum_{j=1}^{i-1} \ang{Y_i,U_j} U_j, \quad U_i = 1_{A_i \cup \dots \cup A_k} \frac{Z_i}{\N{Z_i}}
\quad \mbox{for } 2 \le i \le k.
$$
Then for every $i \in \crl{1, \dots, k}$ satisfying $\mu[A_i] > 0$, $U_1, \dots, U_i$ is an orthonormal basis of $C$ 
on $A_i$. If $k =d$, one obtains from Corollary \ref{corkl} that
$1_{A_d} C = \lin_{A_d} \crl{U_1, \dots, U_d} = 1_{A_d} (L^0)^d$.
If $k < d$, we set $A_{k+1} = \dots = A_d = \emptyset$, and 
$1_{A_d} C = 1_{A_d} (L^0)^d$ holds trivially. By Corollary \ref{corkl}
and Lemma \ref{lemmalargestset}, there exist $V_i \in C$, $i=1, \dots, d$ such that 
$$
1_{A_0} (L^0)^d = \lin_{A_0} \crl{V_1, \dots, V_d}
$$
and
$$
1_{A_i} (L^0)^d = \lin_{A_i} \crl{U_1, \dots, U_i, V_{i+1} \dots, V_d} \quad \mbox{for all } i=1, \dots, d-1.
$$
Set 
$$
X_1 := 1_{A_1 \cup \dots \cup A_d} U_1 + 1_{A_0} \frac{V_1}{\N{V_1}}
$$
and
$$
W_i := V_i - \sum_{j=1}^{i-1} \ang{V_i,X_j} X_j, \quad 
X_i = 1_{A_i \cup \dots \cup A_d} U_i + 1_{A_0 \cup \dots \cup A_{i-1}} \frac{W_i}{\N{W_i}}
\quad \mbox{for } 2 \le i \le d.
$$
Then $X_1, \dots, X_d$ are orthonormal on $\Omega$ such that 
$$1_{A_i} C = \lin_{A_i} \{X_1,\ldots,X_i\}, \quad
1_{A_i} C^{\perp} = \lin_{A_i} \crl{X_{i+1}, \dots, X_d} \quad 
\mbox{for } 1 \le i \le d-1.
$$
It is clear that $C + C^{\perp} = (L^0)^d$, $C \cap C^{\perp} = \crl{0}$ and $C = C^{\perp \perp}$.
\end{proof}

\begin{corollary} \label{cordec}
Let $C$ be a non-empty $\sigma$-stable $L^0$-linear subset of $(L^0)^d$.
Then every $X \in (L^0)^d$ has a unique decomposition $X = Y+Z$ for 
$Y \in C$, $Z \in C^{\perp}$, and $\N{Z} \le \N{X - V}$ for every $V \in C$.
\end{corollary}

\begin{proof}
That $X$ has a unique decomposition $X = Y + Z$, $Y \in C$, $Z \in C^{\perp}$ 
is a consequence of Corollary \ref{corspan}. Moreover, if $V \in C$, then
$$
\N{Z}^2 \le \N{Z}^2 + \N{Y-V}^2 = \N{Z+Y-V}^2 = \N{X-V}^2.
$$
\end{proof}

\setcounter{equation}{0}
\section{Converging sequences, sequential closures and sequential continuity}
\label{sec:limits}

\begin{definition}
We call a subset $C$ of $(L^0)^d$ sequentially closed if it contains every $X \in (L^0)^d$ that is an 
a.e. limit of a sequence $(X_n)_{n \in \mathbb{N}}$ in $C$.
For an arbitrary subset $C$ of $(L^0)^d$ and $A \in {\cal F}_+$, we denote by $\lim_A(C)$ the set
consisting of all a.e. limits of sequences in $1_A C$ and by
$\cl_A(C)$ the smallest sequentially closed subset of $1_A (L^0)^d$ containing $1_A C$.
In the special case $A = \Omega$, we just write $\lim(C)$ and $\cl(C)$, respectively.
\end{definition}

\begin{proposition} \label{proplimC}
For all subsets $C$ of $(L^0)^d$ and $A \in {\cal F}_+$ one has
$\lim_A(C) = \cl_A(C)$.
\end{proposition}

\begin{proof}
It is clear that $\lim_A(C) \subseteq \cl_A(C)$. To show that the two sets are equal,
it is enough to prove that $\lim_A(C)$ is sequentially closed. So
let $(X_n)_{n \in \mathbb{N}}$ be a sequence in $\lim_A(C)$ that converges a.e.
to some $X \in 1_A(L^0)^d$. Since $(\Omega, {\cal F}, \mu)$ is $\sigma$-finite, there exists an increasing
sequence $A_n$, $n \in \mathbb{N}$, of measurable sets such that
$\bigcup_n A_n = A$ and $\mu[A_n] < + \infty$.
For every $n$ there exists a sequence $(Y_m)_{m \in \mathbb{N}}$ in $1_A C$
converging a.e. to $X_n$. Therefore,
$$
\mu[A_n \cap \crl{|Y_m - X_n| > 1/n}] \to 0 \quad \mbox{for } m \to \infty,
$$
and one can choose $m_n \in \mathbb{N}$ such that
$$
\mu[B_n] \le 2^{-n}, \quad \mbox{where } B_n = A_n \cap \crl{|Y_{m_n} - X_n| > 1/n}.
$$
It follows from the Borel--Cantelli lemma that
$\mu \edg{\bigcap_{k \in \mathbb{N}} \bigcup_{n \ge k} B_n} = 0$, which implies
$Y_{m_n} \to X$ a.e. for $n \to \infty$. So $X \in \lim_A(C)$, and the proof is complete.
\end{proof}

\begin{corollary} \label{corAcl}
If $C$ is a stable subset of $(L^0)^d$ and $A \in {\cal F}_+$, then
$$
{\rm lim}_A(C) = 1_A \lim(C) = \cl_A(C) = 1_A \cl(C).
$$
In particular, if $C$ is stable and sequentially closed, then so is $1_A C$.
\end{corollary}

\begin{proof}
${\rm lim}_A(C) = 1_A \lim(C)$ is a consequence of the
stability of $C$. Moreover, it follows from Proposition \ref{proplimC}
that ${\rm lim}_A(C) = \cl_A(C)$ and $\lim(C) = \cl(C)$.
This proves the corollary.
\end{proof}

\begin{corollary} \label{corclstable}
If $C$ is a stable subset of $(L^0)^d$ and $A \in {\cal F}_+$, then $\cl_A(C)$ is $\sigma$-stable.
Moreover, if $C$ is $L^0$-convex, an $L^0$-convex cone, $L^0$-affine or $L^0$-linear, then so is $\cl_A(C)$.
\end{corollary}

\begin{proof}
By Proposition \ref{proplimC}, $\cl_A(C)$ is equal to $\lim_A(C)$. So for all $X,Y \in \cl_A(C)$
there exist sequences $(X_n)_{n \in \mathbb{N}}$ and $(Y_n)_{n \in \mathbb{N}}$
in $1_A C$ such that $X_n \to X$ a.e. and $Y_n \to Y$ a.e. Since for all $B \in {\cal F}$,
$1_B X_n + 1_{B^c} Y_n \in 1_A C$ and $1_B X_n + 1_{B^c} Y_n \to 1_B X + 1_{B^c} Y$ a.e.,
one obtains that $1_B X + 1_{B^c} Y$ belongs to $\lim_A(C) = \cl_A(C)$. 
This shows that $\cl_A(C)$ is stable. Since it is also sequentially closed, it must be 
$\sigma$-stable. The rest of the corollary follows similarly.
\end{proof}

\begin{proposition} \label{propaffclosed}
Every $\sigma$-stable $L^0$-affine subset $C$ of $(L^0)^d$ is sequentially closed.
\end{proposition}

\begin{proof}
If $C$ is empty, the corollary is trivial.
Otherwise, choose $X \in C$. Then $D = C-X$ is a $\sigma$-stable $L^0$-linear subset 
of $(L^0)^d$, and the corollary follows if we can show that $D$ is sequentially closed. So let
$(Y_n)_{n \in \mathbb{N}}$ be a sequence in
$D$ converging a.e. to some $Y \in (L^0)^d$. By Corollary \ref{corspan}, there exist
unique pairwise disjoint sets $A_0, \dots, A_d \in {\cal F}$ satisfying 
$\bigcup_{i=0}^d A_i = \Omega$ and an orthonormal 
basis $X_1, \ldots, X_d$ of $(L^0)^d$ on $\Omega$ such that 
$1_{A_0} D = \crl{0}$ and $1_{A_i} D = \lin_{A_i} \{X_1,\ldots,X_i\}$ for $1 \le i \le d$.
Define $\lambda_n$ and $\lambda$ in $(L^0)^d$ by
$\lambda^j_n := \ang{Y_n,X_j}$ and $\lambda^j := \ang{Y,X_j}$.
Since $Y_n \to Y$ a.e., one has $\lambda^j_n \to \lambda^j$ a.e. In particular,
$\lambda^j = 0$ on $A_i$ such that $i < j$.
This shows that $Y = \sum_j \lambda^j X_j \in D$.
\end{proof}

The following example shows that $L^0$-affine subsets of $(L^0)^d$ that are not
 $\sigma$-stable need not be sequentially closed.

\begin{Example}
Let $\Omega = \mathbb{N}$, ${\cal F} = 2^{\mathbb{N}}$ and
$\mu$ the counting measure. Set $X_n = 1_{\crl{n}} e_1$. Then
$$
\lin(X_n: n \in \mathbb{N}) = \crl{\sum_{n=1}^k \lambda_n X_n : k \in \mathbb{N}, \;
\lambda_1, \dots, \lambda_k \in L^0}
$$
is an $L^0$-linear subset of $(L^0)^d$ that is not $\sigma$-stable, and
$Y_k = \sum_{n=1}^k X_n$ is a sequence in $\lin(X_n: n \in \mathbb{N})$ that converges a.e.
to $\sum_{n \in \mathbb{N}} X_n \notin \lin(X_n: n \in \mathbb{N})$. Note that
$\lin(X_n: n \in \mathbb{N})$ is an $L^0$-submodule of $(L^0)^d$ that is not finitely generated.
\end{Example}

The next result is a conditional version of the Bolzano--Weierstrass theorem. It is already known
(see for instance, Lemma 2 in Kabanov and Stricker (2001) or Lemma 1.63 in F\"ollmer and Schied (2004)).
But since it is important to some of our later results, we give a short proof.
To state the result we need the following definition.

\begin{definition}
We call a subset $C$ of $(L^0)^d$ $L^0$-bounded if $\esssup_{X \in C} \N{X} \in L^0$.
\end{definition}


Note that if $(X_n)_{n \in \mathbb{N}}$ is a sequence in $(L^0)^d$ and $N \in \mathbb{N}({\cal F})$,
$X_N$ can be written as
$$
X_N = \sum_{n \in \mathbb{N}} 1_{\crl{N=n}} X_n.
$$
In particular, $X_N$ is in $(L^0)^d$. Moreover, if all $X_n$ belong to a 
$\sigma$-stable subset $C$ of $(L^0)^d$, then $X_N$ is again in $C$.

\begin{theorem} \label{thmBW} {\bf (Conditional version of the Bolzano--Weierstrass theorem)}\\
Let $(X_n)_{n \in \mathbb{N}}$ be an $L^0$-bounded sequence
in $(L^0)^d$. Then there exists an $X \in (L^0)^d$ and a sequence 
$(N_n)_{n \in \mathbb{N}}$ in $\mathbb{N}({\cal F})$ 
such that $N_{n+1} > N_n$ for all $n \in \mathbb{N}$ and $\lim_{n \to \infty} X_{N_n} = X$ a.e.
\end{theorem}

\begin{proof}
There exists a $Y \in L^0_+$ such that
$\N{X_n} \le Y$ for all $n \in \mathbb{N}$. Therefore,
the a.e. limit $X^1 := \lim_{n \to \infty} \inf_{m \ge n} X^1_m$ exists and is in $L^0$. Define 
$N^1_0 := 0$ and 
$$
N^1_n(\omega) := \min \crl{m \in \mathbb{N} : m > N^1_{n-1}(\omega) \mbox{ and } 
X^1_m(\omega) \le X^1(\omega) + 1/n} \in \mathbb{N}({\cal F}), \quad n \in \mathbb{N}.
$$
Then $N^1_{n+1} > N^1_n$ for all $n \in \mathbb{N}$ and $\lim_{n \to \infty} X^1_{N^1_n} = X^1$ a.e. Now
set $Y^2_n = X^2_{N^1_n}$. Then there exists a sequence
$(M^2_n)_{n \in \mathbb{N}}$ in $\mathbb{N}({\cal F})$ such that
$M^2_{n+1} > M^2_n$ for all $n \in \mathbb{N}$ and
$\lim_{n \to \infty} Y^2_{M^2_n} = X^2 := \lim_{n \to \infty}
\inf_{m \ge n} Y^2_m$ a.e. $N^2_n := N^1_{M^2_n}$, $n \in \mathbb{N}$,
defines a sequence in $\mathbb{N}({\cal F})$ satisfying
$N^2_{n+1} > N^2_n$ for all $n \in \mathbb{N}$, and one has
$\lim_{n \to \infty} X^i_{N^2_n} = X^i$ a.e. for $i=1,2$.
If one continues like this, one obtains $X^1, \dots, X^d \in L^0$ and a sequence $(N_n)_{n \in
\mathbb{N}}$ in $\mathbb{N}({\cal F})$ such that $N_{n+1} > N_n$ 
for all $n \in \mathbb{N}$ and $\lim_{n \to \infty} X_{N_n}
= X = (X^1, \dots, X^d)$ a.e.
\end{proof}

\begin{corollary} \label{corBW}
Let $(X_n)_{n \in \mathbb{N}}$ be a sequence in a sequentially closed
$L^0$-bounded stable subset $C$ of $(L^0)^d$. Then there exists an $X \in C$
and a sequence $(N_n)_{n \in \mathbb{N}}$ in $\mathbb{N}({\cal F})$ such that
$N_{n+1} > N_n$ for all $n \in \mathbb{N}$
and $\lim_{n \to \infty} X_{N_n} = X$ a.e.
\end{corollary}

\begin{proof}
Since $(X_n)_{n \in \mathbb{N}}$ is $L^0$-bounded, it follows
from Theorem \ref{thmBW} that there exists $X \in (L^0)^d$
and a sequence $(N_n)_{n \in \mathbb{N}}$ in $\mathbb{N}({\cal F})$ such that
$N_{n+1} > N_n$ for all $n \in \mathbb{N}$
and $\lim_{n \to \infty} X_{N_n} = X$ a.e. It remains to show that
$X$ belongs to $C$. By Corollary \ref{corclstable} the subset $C$ is $\sigma$-stable. Hence, 
$X_{N_n}$ belongs to $C$ for all $n \in \mathbb{N}$, which implies that $X$ is in $C$ too.
\end{proof}

\begin{corollary} \label{corcl+cl}
Let $C$ and $D$ be non-empty sequentially closed stable subsets of $(L^0)^d$ such that
$D$ is $L^0$-bounded. Then $C+D$ is sequentially closed and stable.
\end{corollary}

\begin{proof}
That $C+D$ is stable is clear. To show that $C+D$ is sequentially closed,
choose a sequence $(X_n)_{n \in \mathbb{N}}$ in $C$ and a sequence $(Y_n)_{n \in \mathbb{N}}$
in $D$ such that $X_n + Y_n \to Z$ a.e. for some $Z \in (L^0)^d$.
By Theorem \ref{thmBW}, there exists $Y \in D$ and a sequence
$(N_n)_{n \in \mathbb{N}}$ in $\mathbb{N}({\cal F})$ such that
$N_{n+1} > N_n$ for all $n \in \mathbb{N}$
and $\lim_{n \to \infty} Y_{N_n} = Y$ a.e. It follows that $\lim_{n \to \infty} X_{N_n} = Z-Y$ a.e.
Since $C$ is and sequentially closed, $Z-Y$ belongs to $C$.
Hence, $Z$ is in $C+D$.
\end{proof}

Another consequence of Theorem \ref{thmBW} is that conditional Cauchy sequences converge
if they are defined as follows:

\begin{definition}
We call a sequence $(X_n)_{n \in \mathbb{N}}$ in $(L^0)^d$ $L^0$-Cauchy if
for every $\varepsilon \in L^0_{++}$ there exists an $N_0 \in \mathbb{N}({\cal F})$ such that
$\N{X_{N_1} - X_{N_2}} \le \varepsilon$ for all $N_1, N_2 \in \mathbb{N}({\cal F})$
satisfying $N_1, N_2 \ge N_0$.
\end{definition}

\begin{theorem} \label{thmCauchy}
Every $L^0$-Cauchy sequence $(X_n)_{n \in \mathbb{N}}$ in $(L^0)^d$
converges a.e. to some $X \in (L^0)^d$.
\end{theorem}

\begin{proof}
Choose $N_0 \in \mathbb{N}({\cal F})$ such that
$\N{X_{N_1} - X_{N_2}} \le 1$ for all $N_1,N_2 \in \mathbb{N}({\cal F})$
satisfying $N_1, N_2 \ge N_0$. Then
$$
\N{X_n} \le 1+ \sum_{m \in \mathbb{N}} 1_{\crl{m \le N_0}}
\N{X_m}  \in L^0
$$
for all $n \in \mathbb{N}$. So it follows from Theorem
\ref{thmBW} that there exist $X \in (L^0)^d$ and
a sequence $(N_n)_{n \in \mathbb{N}}$ in $\mathbb{N}({\cal F})$ such that
$N_{n+1} > N_n$ for all $n \in \mathbb{N}$ and
$\lim_{n \to \infty} X_{N_n} = X$ a.e. But since
$(X_n)_{n \in \mathbb{N}}$ is $L^0$-Cauchy, one has $\lim_{n \to \infty} X_n  = X$ a.e.
\end{proof}

The following result gives necessary and sufficient conditions for a sequentially closed 
$L^0$-convex subset of $(L^0)^d$ to be $L^0$-bounded.

\begin{theorem}\label{thmbounded}
Let $C$ be a sequentially closed $L^0$-convex subset of $(L^0)^d$ containing $0$.
Then $C$ is $L^0$-bounded if and only if for any $X\in C \setminus \crl{0}$ there exists a
$k \in {\mathbb N}$ such that $k X\not\in C$.
\end{theorem}

\begin{proof}
Suppose that $C$ is $L^0$-bounded. Then for every $0 \neq X\in C$, there 
exists a $k \in {\mathbb N}$ such that
$\mu\left[\|k X\|> \mathop{\rm ess\,sup}_{Y\in C} \|Y\|\right]>0$,
and therefore $k X\notin C$.

Conversely, suppose that $C$ is not $L^0$-bounded. The sequence
$$
A_n:=\mathop{\rm ess\,sup}\crl{B\in{\cal F} : \|X\| \ge n\mbox{ on }B\mbox{ for some }X\in C}, 
\quad n \in \mathbb{N} \cup \crl{0},
$$
is decreasing with limit $A:=\bigcap_n A_n$. One must have
$\mu[A]>0$, since otherwise, $\|X\|\le\sum_{n\in{\mathbb N}} n 1_{\{A_n^c\setminus A^c_{n-1}\}}\in L^0$ for all $X\in C$. 
Since $C$ is sequentially closed, $L^0$-convex and therefore stable, it is  
$\sigma$-stable. It follows that there exists 
a sequence $(X_n)_{n \in {\mathbb N}}$ in $C$ such that $\|X_n\|\ge n$ on $A$. 
Since the sequence $Y_n = 1_A X_n/{\|X_n\|}$ is $L^0$-bounded, it follows from 
Theorem \ref{thmBW} that there exists $Y\in (L^0)^d$ and a sequence $(N_n)_{n\in{\mathbb N}}$ in 
${\mathbb N}({\cal F})$ such that $N_{n+1} > N_n$ and $\lim_{n\to\infty} Y_{N_n}=Y$ a.e. 
Obviously, $1_A \N{Y} = 1_A$, and in particular, $Y \neq 0$.
Since $C$ is $L^0$-convex, sequentially closed and contains $0$, one has for all 
$n \ge k$, 
$$
k Y_{N_n}= 1_A \frac{k}{\|X_{N_n}\|} X_{N_n} \in C.
$$
But $\lim_{n\to\infty} k Y_{N_n} = k Y$. So $k Y\in C$ for all $k \in{\mathbb N}$. 
\end{proof}

\begin{definition}
Let $C$ be a non-empty subset of $(L^0)^d$ and $k \in \mathbb{N}$. 
We call a function $f : C \to (L^0)^k$ 
\begin{itemize}
\item sequentially continuous at $X \in C$ if $f(X_n) \to f(X)$ a.e. for every sequence $(X_n)_{n \in \mathbb{N}}$ in $C$
converging to $X$ a.e.;
\item sequentially continuous if it is sequentially continuous at every $X \in C$;
\item
$L^0$-affine if 
$f(\lambda X + (1-\lambda)Y) = \lambda f(X) + (1-\lambda) f(Y)$
for all $X,Y \in (L^0)^d$ and $\lambda \in L^0$ such that $\lambda X + (1-\lambda)Y \in C$;
\item
$L^0$-linear if $f(\lambda X + Y) = \lambda f(X) + f(Y)$ for all
$X,Y \in (L^0)^d$ and $\lambda \in L^0$ such that $\lambda X + Y \in C$.
\item
We define the conditional norm of $f$ by
$\N{f} := \esssup_{X \in C, \, \N{X} \le 1} \N{f(X)} \in \overline{L}$.
\end{itemize}
\end{definition}

\begin{proposition} \label{propLipschitz}
Let $C$ be a non-empty $\sigma$-stable $L^0$-linear subset of $(L^0)^d$.
Then $\N{f} \in L^0_+$ for every $L^0$-linear function $f : C \to (L^0)^k$, $k \in \mathbb{N}$. 
\end{proposition}

\begin{proof}
By Corollary \ref{corspan}, there exist unique pairwise disjoint sets $A_0, \dots, A_d \in {\cal F}$ satisfying 
$\bigcup_{i=0}^d A_i = \Omega$ and an orthonormal 
basis $X_1, \ldots, X_d$ of $(L^0)^d$ on $\Omega$ such that 
$1_{A_0} C = \crl{0}$ and $1_{A_i} C = \lin_{A_i} \{X_1,\ldots,X_i\}$ for $1 \le i \le d$.
For every $X \in C$ there exists a unique 
$\lambda \in (L^0)^d$ such that $X = \sum_{j=1}^d \lambda_j X_j$. 
On the set $A_0$ one has $f(X) = X = 0$, and on $A_i$ for $1 \le i \le d$,
$\N{X} = \brak{\sum_{j=1}^i \lambda^2_j}^{1/2}$ as well as
$$
\N{f(X)} = \N{\sum_{j=1}^i \lambda_j f(X_j)} \le 
\sum_{j=1}^i |\lambda_j| \N{f(X_j)} \le 
\brak{\sum_{j=1}^i \lambda_j^2}^{1/2} \brak{ \sum_{j=1}^i \N{f(X_j)}^2}^{1/2}.
$$
Therefore, $\N{f} \le \sum_{i=1}^d 1_{A_i} \brak{\sum_{j=1}^i \N{f(X_j)}^2}^{1/2}$.
\end{proof}

\begin{corollary} \label{coraffcont}
Let $C$ be a non-empty $\sigma$-stable $L^0$-affine subset of $(L^0)^d$.
Then every $L^0$-affine function $f : C \to (L^0)^k$, $k \in \mathbb{N}$, is sequentially continuous.
\end{corollary}

\begin{proof}
Choose an $X_0 \in C$. Then $D = C -X_0$ is a non-empty $\sigma$-stable $L^0$-linear subset of $(L^0)^d$
and $g(X) = f(X + X_0) - f(X_0)$ is an $L^0$-linear function on $D$. By
Proposition \ref{propLipschitz}, one has $\N{g} \in L^0_+$. Moreover, 
$\N{f(X) - f(Y)} = \N{g(X-Y)} \le \N{g} \, \N{X-Y}$, and it follows that
$f$ is sequentially continuous.
\end{proof}

\begin{corollary} \label{corfclosed}
Let $C$ be a non-empty sequentially closed subset of a non-empty 
$\sigma$-stable $L^0$-affine subset $D$ of $(L^0)^d$.
Then for every injective $L^0$-affine function $f : D \to (L^0)^k$, $k \in \mathbb{N}$,
$f(C)$ is a sequentially closed subset of $(L^0)^k$.
\end{corollary}

\begin{proof}
Pick an $X_0 \in C$. The corollary follows if we can show
that $f(C) - f(X_0)$ is sequentially closed. So by replacing $C$ with $C - X_0$, $D$ with $D-X_0$
and $f$ with $f(X+X_0) - f(X_0)$, one can assume
that $X_0 =0$, $D$ is a $\sigma$-stable $L^0$-linear subset of $(L^0)^d$ and $f$ is injective
$L^0$-linear. By Corollary \ref{coraffcont}, $f$ is sequentially continuous. 
Therefore, $f(D)$ is a non-empty $\sigma$-stable $L^0$-linear subset of $(L^0)^k$, and it 
follows from Proposition \ref{propaffclosed} that it is sequentially closed. 
Since $f^{-1} : f(D) \to D$ is again $L^0$-linear, it is also sequentially continuous. So if $(Y_n)_{n \in \mathbb{N}}$ is a 
sequence in $f(C)$ converging a.e. to some $Y \in (L^0)^k$, then $Y \in f(D)$ and $f^{-1}(Y_n)$ is a sequence in $C$ 
converging a.e. to $f^{-1}(Y) \in D$. It follows that $f^{-1}(Y) \in C$ and $Y = f(f^{-1}(Y)) \in f(C)$.
\end{proof}

\begin{lemma}
Let $C$ be a non-empty $\sigma$-stable $L^0$-linear subset of $(L^0)^d$ and $k \in \mathbb{N}$.
Then every $L^0$-linear function $f : C \to (L^0)^k$ has an $L^0$-linear extension 
$F : (L^0)^d \to (L^0)^k$ such that $\N{f} = \N{F}$.
\end{lemma}

\begin{proof}
By Corollary \ref{cordec}, every 
$X \in (L^0)^d$ has a unique decomposition $X = Y+Z$ such that $Y \in C$ and $Z \in C^{\perp}$.
$F(X) := f(Y)$ defines an $L^0$-linear extension of $f$ to $(L^0)^d$ such that $\N{f} = \N{F}$.
\end{proof}

\setcounter{equation}{0}
\section{Conditional optimization}
\label{sec:opti}

\begin{definition}
Let $C$ be a non-empty subset of $(L^0)^d$. We call a function $f : C \to L$
\begin{itemize}
\item sequentially lsc (lower semicontinuous) at $X \in C$ 
if $f(X) \le \liminf_{n \to \infty} f(X_n)$ for every sequence
$(X_n)_{n \in \mathbb{N}}$ in $C$ with a.e. limit $X$;
\item sequentially lsc if it is sequentially lsc at every $X \in C$;
\item sequentially usc (upper semicontinuous) at $X \in C$ if $-f$ is sequentially lsc at $X$;
\item sequentially usc if it is sequentially usc at every $X \in C$;
\item sequentially continuous at $X \in C$ if it is sequentially lsc and usc at $X$;
\item sequentially continuous if it is sequentially continuous at every $X \in C$.
\end{itemize}
\end{definition}

In the following definition $+ \infty - \infty$ is understood as $+\infty$ and $0 \cdot (\pm \infty)$ as $0$.

\begin{definition} \label{defsimpleconvex}
Let $f : C \to L$ be a function on a non-empty subset $C$ of $(L^0)^d$.
\begin{itemize}
\item If $C$ is stable, we call $f$ stable if 
$$
f(1_A X + 1_{A^c} Y) = 1_A f(X) + 1_{A^c} f(Y)
$$
for all $X,Y \in C$ and $A \in {\cal F}_+$;
\item If $C$ is $L^0$-convex, we call $f$ $L^0$-convex if 
$$
f(\lambda X + (1-\lambda)Y) \le \lambda f(X) + (1-\lambda)f(Y) 
$$
for all $X,Y \in C$ and $\lambda \in L^0$ such that $0 \le \lambda \le 1$;
\item If $C$ is $L^0$-convex, we call $f$ strictly $L^0$-convex if 
$$
f(\lambda X + (1-\lambda) Y) < \lambda f(X) + (1-\lambda)f(Y)
\quad \mbox{on the set } \crl{X \neq \lambda X + (1-\lambda) Y) \neq Y}
$$
for all $X,Y \in C$ and $\lambda \in L^0$ such that $0 \le \lambda \le 1$.
\end{itemize}
\end{definition}

\begin{lemma} \label{lemmastablef}
Let $f : C \to L$ be an $L^0$-convex function on an $L^0$-convex subset $C$ of $(L^0)^d$.
Then $f$ is also stable.
\end{lemma}

\begin{proof}
Let $X,Y \in C$ and $A \in {\cal F}_+$. Denote $Z = 1_A X + 1_{A^c} Y$. Then one has 
$1_A f(Z) \le 1_A f(X)$ and $1_A f(X) = 1_A f(1_A Z + 1_{A^c}X) \le 1_A f(Z)$.
This shows that $1_A f(Z) = 1_A f(X)$. Analogously, one obtains 
$1_{A^c} f(Z) = 1_{A^c} f(Y)$ and therefore $f(Z) = 1_Af(X) + 1_{A^c} f(Y)$.
\end{proof}

\begin{theorem} \label{thmopt}
Let $C$ be a sequentially closed stable subset of $(L^0)^d$ and
$f : C \to \overline{L}$ a sequentially lsc stable function.
Assume there exists an $X_0 \in C$ such that the set
$$\crl{X \in C : f(X) \le f(X_0)}$$ is $L^0$-bounded.
Then there exists an $\hat{X} \in C$ such that
$$
f(\hat{X}) = \essinf_{X \in C} f(X).
$$
If $C$ and $f$ are $L^0$-convex, then the set
$$
\crl{X \in C : f(X) = f(\hat{X})}
$$
is $L^0$-convex. If in addition, $f$ is strictly $L^0$-convex, then
$$
\crl{X \in C : f(X) = f(\hat{X})} = \crl{\hat{X}}.
$$
\end{theorem}

\begin{proof}
The set $D := \crl{X \in C : f(X) \le f(X_0)}$ is sequentially closed, 
stable and $L^0$-bounded. It follows that
$\crl{f(X) : X \in D}$ is directed downwards. Therefore, there
exists a sequence $(X_n)_{n \in \mathbb{N}}$ in $D$ such that
$f(X_n)$ decreases a.e. to $I : = \mathop{\rm ess\,inf}_{X \in D} f(X)$.
By Corollary \ref{corBW}, there exists
a sequence $(N_n)_{n \in \mathbb{N}}$ in $\mathbb{N}({\cal F})$ such that
$N_{n+1} > N_n$ for all $n \in \mathbb{N}$
and $\lim_{n \to \infty} X_{N_n} = \hat{X}$ a.e. for some $\hat{X} \in D$.
Since $X_{N_n}$ belongs to $D$ and
$$
f(X_{N_n}) = \sum_{m \ge n} 1_{\crl{N_n = m}} f(X_m) \le f(X_n) \quad \mbox{for all } n,
$$
one obtains from the $L^0$-lower semicontinuity of $f$ that
$$
f(\hat{X}) \le \liminf_{n \to \infty} f(X_{N_n}) \le \lim_{n \to \infty} f(X_n) = I.
$$
This shows the first part of the theorem. 
That $\crl{X \in C : f(X) = f(\hat{X})}$ is $L^0$-convex if
$C$ and $f$ are $L^0$-convex, is clear. Finally, assume $C$ is $L^0$-convex and $f$ strictly $L^0$-convex.
Then if there exists an $X$ in $C$ such that $f(X) = f(\hat{X})$, one has
$$
f \brak{\frac{X + \hat{X}}{2}} < \frac{f(X) + f(\hat{X})}{2}
$$
on the set $\crl{X \neq \hat{X}}$. It follows that $\mu[X \neq \hat{X}] = 0$.
\end{proof}

\begin{corollary} \label{cordist}
Let $C$ and $D$ be non-empty sequentially closed stable subsets of
$L^0({\cal F})^d$ such that $D$ is $L^0$-bounded. Then there exist
$\hat{X} \in C$ and $\hat{Y} \in D$ such that
\be \label{minXY}
\N{\hat{X}- \hat{Y}} = \essinf_{X \in C, \, Y \in D} \N{X -Y}.
\ee
If in addition, $C$ and $D$ are $L^0$-convex, then $\hat{X} - \hat{Y}$ is unique.
\end{corollary}

\begin{proof}
By Corollary \ref{corcl+cl}, the set $E = C-D$ is sequentially closed and stable.
Moreover, $Z \mapsto \N{Z}$ is a sequentially continuous $L^0$-convex function from $E$ to $L^0$, and for
every $Z_0 \in E$, the set $\crl{Z \in E : \N{Z} \le \N{Z_0}}$ is $L^0$-bounded.
So one obtains from Theorem \ref{thmopt} that there exists a $\hat{Z} \in E$ such that
$\N{\hat{Z}} = \essinf_{Z \in E} \N{Z}$. This shows that there exist
$\hat{X} \in C$ and $\hat{Y} \in D$ satisfying \eqref{minXY}.
If $C$ and $D$ are $L^0$-convex, then so is $E$, and for every $Z \in E$
satisfying $\N{Z} = \N{\hat{Z}}$, one has $(Z+\hat{Z})/2 \in E$ and
$\N{(Z+\hat{Z})/2} < \N{\hat{Z}}$ on the set $\crl{Z \neq \hat{Z}}$. It follows that
$\mu[Z \neq \hat{Z}] = 0$, and the proof is complete.
\end{proof}

\setcounter{equation}{0}
\section{Interior, relative interior and $L^0$-open sets}
\label{sec:open}

\begin{definition}
Let $C$ be a non-empty subset of $(L^0)^d$ and $A \in {\cal F}_+$. 

\begin{itemize}
\item
For $X \in (L^0)^d$ and $\varepsilon \in L^0_{++}$, we denote 
$$
B^{\varepsilon}_A(X) := \crl{Y \in 1_A(L^0)^d : 1_A \N{Y-X} \le \varepsilon}.
$$
\item
The interior $\inn_A (C)$ of $C$ on $A$ consists of 
elements $X \in 1_A C$ for which there exists an 
$\varepsilon \in L^0_{++}$ such that $B^{\varepsilon}_A(X) \subseteq 1_A C$.
If $A = \Omega$, we just write $\inn(C)$ for $\inn_A(C)$.
\item
The relative interior $\ri_A (C)$ of $C$ on $A$ consists of 
elements $X \in 1_A C$ for which there exists an $\varepsilon \in L^0_{++}$  such that 
$B^{\varepsilon}_A(X) \cap \aff_A(C) \subseteq 1_A(C)$.
If $A = \Omega$, we write $\ri(C)$ instead of $\ri_A(C)$.
\item
We say $C$ is $L^0$-open on $A$ if $1_A C = \inn_A(C)$. We call it $L^0$-open if it is $L^0$-open
on $\Omega$.
\end{itemize}
\end{definition}

Note that one always has $1_A \inn(C) \subseteq \inn_A(C)$ but not necessarily the other way around.
The collection of all $L^0$-open subsets of $(L^0)^d$ forms a topology. It is studied in 
Filipovi\'c et al. (2009) and is related to $(\varepsilon,\lambda)$-topologies 
on random locally convex modules (see Guo, 2010). We point out that sequentially closed sets 
in $(L^0)^d$ are different from complements of $L^0$-open sets. But 
one has the following relation between the two:

\begin{lemma}\label{lemmaopenclosed}
Let $C$ be a $\sigma$-stable subset of $(L^0)^d$. Then 
$\cl(C) \cap \inn(C^c) = \emptyset$.
\end{lemma}

\begin{proof}
Assume $X \in \cl(C) \cap \inn(C^c)$. By Proposition \ref{proplimC}, there exists a sequence 
$(X_n)_{n \in \mathbb{N}}$ in $C$ such that $X_n \to X$ a.e. On the other hand, there is an
$\varepsilon \in L^0_{++}$ such that $Y \in C^c$ for every $Y \in (L^0)^d$ satisfying $\N{X-Y} \le \varepsilon$.
$N(\omega) := \min \crl{n \in \mathbb{N} : \N{X_n(\omega) - X(\omega)} \le \varepsilon(\omega)}$
is an element of $\mathbb{N}({\cal F})$, and since $C$ is $\sigma$-stable, $X_{N}$ belongs to $C$.
But at the same time one has $\N{X_N-X} \le \varepsilon$, implying $X_N \in C^c$.
This yields a contradiction. So $\cl(C) \cap \inn(C^c) = \emptyset$.
\end{proof}

\begin{lemma} \label{lemmalambdaint} 
Let $C$ be a non-empty $L^0$-convex subset of $(L^0)^d$, $A \in {\cal F}_+$ and
$\lambda \in L^0$ such that $0 < \lambda \le 1$. Then
\be \label{innCC}
\lambda X + (1-\lambda) Y \in \inn_A(C) \quad 
\mbox{for all } X \in \inn_A(C), \;Y \in 1_A C
\ee
and
\be \label{riCC}
\lambda X + (1-\lambda) Y \in \ri_A(C) \quad 
\mbox{for all } X \in \ri_A(C), \; Y \in 1_A C.
\ee
If in addition, $C$ is $\sigma$-stable, then \eqref{innCC} and \eqref{riCC} also hold for
$Y \in \cl_A(C)$.
\end{lemma}

\begin{proof}
Let $X \in \inn_A(C)$ and $Y \in 1_A C$. 
There exists an $\varepsilon \in L^0_{++}$ such that $B^{\varepsilon}_A(X)$ is contained in $1_AC$.
So 
$$
\lambda X + (1-\lambda)Y + Z = \lambda(X + Z/\lambda) + (1-\lambda) Y \subseteq 1_A C 
$$
for all $Z \in B^{\varepsilon \lambda}_A(0)$.
This shows \eqref{innCC}. 

To prove \eqref{riCC}, we assume that $X \in \ri_A(C)$ and $Y \in 1_A C$.
There exists an $\varepsilon \in L^0_{++}$ 
such that $B^{\varepsilon}_A(X) \cap \aff_A(C) \subseteq 1_A C$. Choose 
$Z \in B^{\varepsilon \lambda}_A(0)$ such that 
$$
\lambda X + (1-\lambda)Y + Z \in \aff_A(C).
$$
Then $X + Z/\lambda \in \aff_A(C)$, and therefore $X + Z/\lambda \in 1_A C$.
It follows that
$$
\lambda X + (1-\lambda)Y + Z = \lambda(X + Z/\lambda) + (1-\lambda) Y \subseteq 1_A C.
$$
This shows \eqref{riCC}.

If $C$ is $\sigma$-stable, $X \in \inn_A(C)$ and $Y \in \cl_A(C)$, there exists an
$\varepsilon \in L^0_{++}$ such that $B^{2 \varepsilon}_A(X) \subseteq 1_A C$. 
From Lemma \ref{proplimC} we know that there exists a sequence 
$(Y_n)_{n \in \mathbb{N}}$ in $1_A C$ converging a.e. to $Y$.
$N(\omega) := \min \crl{n \in \mathbb{N} : (1-\lambda(\omega)) \N{Y(\omega)-Y_n(\omega)} \le \lambda(\omega)
\varepsilon(\omega)}$ belongs to 
$\mathbb{N}({\cal F})$, and $Y_N$ is an element of $C$ satisfying $(1-\lambda) \N{Y-Y_N} \le \lambda \varepsilon$.
So for $Z \in B^{\lambda \varepsilon}_A(0)$, one has
$$
\lambda X + (1-\lambda)Y + Z =
\lambda \brak{X + \frac{(1-\lambda)}{\lambda} (Y-Y_N) + \frac{1}{\lambda} Z} + (1-\lambda) Y_N \in 1_A C,
$$
which shows that $\lambda X + (1-\lambda)Y \in \inn_A(C)$.

If $X$ is in $\ri_A(C)$ instead of $\inn_A(C)$, there exists an
$\varepsilon \in L^0_{++}$ such that $B^{2 \varepsilon}_A(X) \cap \aff_A(C) \subseteq 1_A C$. 
Let $Z \in B^{\lambda \varepsilon}_A(0)$ such that 
$$
\lambda X + (1-\lambda)Y + Z \in \aff_A(C),
$$
then 
$$
X + \frac{(1-\lambda)}{\lambda} (Y-Y_N) + \frac{1}{\lambda} Z \in \aff_A(C).
$$
Hence
$$
X + \frac{(1-\lambda)}{\lambda} (Y-Y_N) + \frac{1}{\lambda} Z \in 1_A C,
$$
and it follows that
$$
\lambda X + (1-\lambda)Y + Z =
\lambda \brak{X + \frac{(1-\lambda)}{\lambda} (Y-Y_N) + \frac{1}{\lambda} Z} + (1-\lambda) Y_N \in 1_A C.
$$
So $\lambda X + (1-\lambda)Y \in \ri_A(C)$, and the proof is complete.
\end{proof}

\begin{corollary}
Let $C$ be an $L^0$-convex subset of $(L^0)^d$ and $A \in {\cal F}_+$. 
Then $\inn_A(C)$ and $\ri_A(C)$ are again $L^0$-convex.
\end{corollary}

\begin{proof}
Since $C$ is stable, it follows from Lemma \ref{lemmalambdaint} that for 
$X,Y \in \inn_A(C)$ and $\lambda \in L^0$ satisfying $0 \le \lambda \le 1$, one has
$$
\lambda X + (1-\lambda) Y = 
1_{\crl{\lambda > 0}} (\lambda X + (1-\lambda) Y) + 1_{\crl{\lambda = 0}} Y \in \inn_A(C).
$$
This shows that $\inn_A(C)$ is $L^0$-convex. The same argument shows that 
$\ri_A(C)$ is $L^0$-convex.
\end{proof}


\begin{definition}
Let $A \in {\cal F}_+$. We call a subset $C$ of $(L^0)^d$  
\begin{itemize}
\item an $L^0$-hyperplane on $A$ if $1_A C= \crl{X \in 1_A (L^0)^d : \ang{X,Z} = V}$
\item an $L^0$-halfspace on $A$ if $1_A C= \crl{X \in 1_A(L^0)^d : \ang{X,Z} \ge V}$
\end{itemize}
for some $V \in 1_A L^0$ and $Z \in 1_A (L^0)^d$ such that $\N{Z} > 0$ on $A$.
\end{definition}

\begin{lemma} \label{lemmahalfspace}
A subset $C$ of $(L^0)^d$ is an $L^0$-hyperplane on $A \in {\cal F}_+$ 
if and only if there exist $X_0 \in 1_A(L^0)^d$ and an orthonormal 
basis $X_1, \ldots, X_d$ of $(L^0)^d$ on $A$ such that 
\be \label{hyperplane}
1_A C = \crl{X_0 + \sum_{i=1}^{d-1} \lambda_i X_i : \lambda_i \in 1_A L^0}.
\ee
Similarly, $C$ is an $L^0$-halfspace on $A \in {\cal F}_+$ if and only if
there exist $X_0 \in 1_A(L^0)^d$ and an orthonormal 
basis $X_1, \ldots, X_d$ of $(L^0)^d$ on $A$ such that 
\be \label{halfspace}
1_A C = \crl{X_0 + \sum_{i=1}^{d} \lambda_i X_i : \lambda_i \in 1_A L^0, \, \lambda_d \ge 0}.
\ee
\end{lemma}

\begin{proof}
If $1_A C$ is of the form \eqref{hyperplane}, then $1_A C = \crl{X \in 1_A (L^0)^d : \ang{X,X_d} =  \ang{X_0,X_d}}$.
Now assume that $1_A C= \crl{X \in 1_A (L^0)^d : \ang{X,Z} = V}$
for some $V \in 1_A L^0$ and $Z \in 1_A (L^0)^d$ such that $\N{Z} > 0$ on $A$. By 
Corollary \ref{corspan}, there exists an orthonormal
basis $X_1, \dots, X_d$ of $(L^0)^d$ on $A$ such that $1_A Z^{\perp} = \lin_A \crl{X_1, \dots, X_{d-1}}$
and $X_d = 1_A Z/\N{Z}$. Choose $X_0 \in 1_A(L^0)^d$ such that $\ang{X_0,Z} = V$.
Then $1_A C$ is of the form \eqref{hyperplane}. That $C$ is an $L^0$-halfspace on $A \in {\cal F}_+$ 
if and only if $1_A C$ is of the form \eqref{halfspace} follows similarly.
\end{proof}

\begin{lemma} \label{lemmaintaff}
Let $C$ be a $\sigma$-stable $L^0$-convex subset of $(L^0)^d$ and $A \in {\cal F}_+$. Then
$\inn_A(C) \not= \emptyset$ if and only if $\aff_A(C) = 1_A(L^0)^d$.
\end{lemma}

\begin{proof}
Let us first assume that $X_0 \in \inn_A(C)$. Then $0 \in \inn_A(C-X_0)$, and it follows that 
$$
\aff_A(C) = \aff_A(C-X_0) +X_0 = \lin_A(C-X_0) + X_0 = 1_A(L^0)^d + X_0 = 1_A(L^0)^d.
$$
On the other hand, if $\aff_A(C) = 1_A(L^0)^d$, choose $X_0 \in 1_A C$. Then 
$$
\lin_A(C-X_0) = \aff_A(C-X_0) = \aff_A(C) - X_0 = 1_A(L^0)^d.
$$
So it follows from Theorem \ref{thmspan} that there exist $X_1, \dots, X_d$ in $1_A C$ such that 
$X_i-X_0$, $i = 1, \dots, d,$ form a basis of $(L^0)^d$ on $A$. Set
$$
\hat{X} := \frac{1}{d+1} \sum_{i=0}^d X_i.
$$
It follows from Corollary \ref{corspan} and Lemma \ref{lemmahalfspace} that
for every $i=0, \dots, d$, there exist $V_i \in L^0$ and $Z_i \in (L^0)^d$ such that for all $j \not= i$,
$$
\ang{\hat{X},Z_i} > V_i = \ang{X_j,Z_i} \mbox{ on } A.
$$
This shows that $\hat{X} \in \inn_A  \crl{X \in 1_A(L^0)^d : \ang{X,Z_i} \ge V_i}$ for all $i$, which implies 
$\hat{X} \in \inn_A(C)$ since
$$
\bigcap_{i=0}^d \crl{X \in 1_A(L^0)^d : \ang{X,Z_i} \ge V_i} = \conv_A \crl{X_0, \dots, X_d} \subseteq 1_A C.
$$
\end{proof}

\setcounter{equation}{0}
\section{Separation by $L^0$-hyperplanes}
\label{sec:sep}

In this section we prove results on the separation of two $L^0$-convex sets in
$(L^0)^d$ by an $L^0$-hyperplane. As a corollary we obtain a version of the 
Hahn--Banach extension theorem. Hahn--Banach extension and separation results
have been proved in more general modules; see e.g., Filipovi\'c et al. (2009), Guo (2010) 
and the references therein. However, due to the special form of $(L^0)^d$, we here are able to derive 
analogs of results that hold in $\mathbb{R}^d$ but not in infinite-dimensional vector spaces. Moreover, we do not 
need Zorn's lemma or the axiom of choice.

\begin{theorem} \label{thmsepcl} {\bf (Strong separation)}\\
Let $C$ and $D$ be non-empty $L^0$-convex subsets of $(L^0)^d$. Then there exists
$Z \in (L^0)^d$ such that
\be \label{sepcl}
\essinf_{X \in C} \ang{X,Z} > \esssup_{Y \in D} \ang{Y,Z}
\ee
if and only if $0 \notin \cl_A(C-D)$ for all $A \in {\cal F}_+$.
\end{theorem}

\begin{proof}
Let us first assume that there exists an $A \in {\cal F}_+$ such that $0 \in \cl_A(C-D)$.
From Proposition \ref{proplimC} we know that $\cl_A(C-D) = \lim_A(C-D)$. 
So there exists a sequence $(X_n)_{n \in \mathbb{N}}$ in $1_A(C-D)$ such that 
$X_n \to 0$ a.e. It follows that there can exist no $Z \in (L^0)^d$ satisfying \eqref{sepcl}.

Now assume $0 \notin \cl_A(C-D)$ for all $A \in {\cal F}_+$.
It follows from Corollary \ref{corclstable} that $\cl(C-D)$ is $L^0$-convex.
So one obtains from Corollary \ref{cordist} that there exists a $Z \in \cl(C-D)$ such that
$$
\|Z\|^2 \le \|(1-\lambda)Z + \lambda W\|^2\\
= \|Z\|^2 + 2 \lambda \ang{Z,W-Z} + \lambda^2 \|W - Z\|^2
$$
for all $W \in \cl(C-D)$ and $\lambda \in L^0$ such that $0 < \lambda \le 1$.
Division by $2 \lambda$ and sending $\lambda$ to $0$ yields
$\ang{W,Z} \ge \|Z\|^2$. In particular,
$$
\ang{W,Z} \ge \N{Z}^2 \quad \mbox{for all } W \in C-D,
$$
and therefore,
$$
\essinf_{X \in C} \ang{X,Z} \ge \esssup_{Y \in D} \ang{Y,Z} + \N{Z}^2.
$$
It remains to show that $\|Z\| > 0$. But if this were not the case, the set
$A = \crl{Z=0}$ would belong to ${\cal F}_+$ and
$1_A Z = 0$. However, by assumption and Corollary \ref{corAcl}, one has
$0 \notin \cl_A(C-D) = 1_A \cl(C-D)$ for all $A \in {\cal F}_+$, a contradiction.
\end{proof}

\begin{corollary} \label{corsepcl}
Let $C$ and $D$ be non-empty sequentially closed $L^0$-convex subsets of
$(L^0)^d$ such that $D$ is $L^0$-bounded and $1_A C$ is disjoint from
$1_A D$ for all $A \in {\cal F}_+$. Then there exists a $Z \in (L^0)^d$ such that
$$
\essinf_{X \in C} \ang{X,Z} > \esssup_{Y \in D} \ang{Y,Z}.
$$
\end{corollary}

\begin{proof}
$C-D$ is a non-empty $L^0$-convex set, which by Corollary \ref{corcl+cl}
is sequentially closed. It follows from the assumptions that
$0 \notin 1_A (C-D)$ for all $A \in {\cal F}_+$, and we know from Corollary \ref{corAcl} that
$1_A(C-D) = \cl_A(C-D)$. So the corollary is a consequence of Theorem \ref{thmsepcl}.
\end{proof}

\begin{lemma} \label{lemmasepcone}
Let $C$ be a non-empty $\sigma$-stable $L^0$-convex cone in $(L^0)^d$ such that 
$1_A C \not= 1_A(L^0)^d$ for all $A \in {\cal F}_+$.
Then there exists a $Z \in (L^0)^d$ such that
\be \label{conesep}
\N{Z} > 0 \quad \mbox{and} \quad \essinf_{X \in C} \ang{X,Z} \ge 0.
\ee
\end{lemma}

\begin{proof}
If $C = \crl{0}$, the lemma is clear. Otherwise one obtains from Theorem \ref{thmspan} 
that there exist $A \in {\cal F}$ and $X_1, \dots, X_{d-1} \in C$ such that
$\lin_A(C) = \lin_A (L^0)^d$ and $\lin_{A^c} (C) \subseteq \lin_{A^c} \crl{X_1, \dots,X_{d-1}}$.
By Corollary \ref{corspan}, there exists $W \in \lin_{A^c} \crl{X_1, \dots,X_{d-1}}^{\perp}$ 
such that $\N{W} > 0$ on $A^c$. If $\mu[A] = 0$, then $Z = W$ satisfies \eqref{conesep}, 
and the proof is complete. If $\mu[A] >0$, one notes that since $C$ is an $L^0$-convex cone, one has
$\aff_A(C) = \lin_A(C) = 1_A (L^0)^d$. It follows from
Lemma \ref{lemmaintaff} that there exists a $Y \in \inn_A(C)$.
Then $1_B Y \in \inn_B(C)$ for every subset $B \in {\cal F}_+$ of $A$. But this implies
that $-1_B Y$ cannot be in $\cl_B(C)$. Otherwise it would follow from 
Lemma \ref{lemmalambdaint} that $0$ belongs to $\inn_B(C)$, implying that $1_B C = 1_B (L^0)^d$ 
and contradicting the assumptions. So Theorem \ref{thmsepcl} applied to $1_A C$ and $\crl{-Y}$ 
viewed as subsets of $1_A(L^0)^d$ yields a $V \in 1_A (L^0)^d$ such that 
$$
\essinf_{X \in 1_A C} \ang{X,V} > \ang{-Y,V} \quad \mbox{on } A.
$$
Since $C$ is an $L^0$-convex cone, $Z = 1_A V + 1_{A^c} W$ satisfies condition \eqref{conesep}.
\end{proof}

\begin{theorem} \label{thmweaksep} {\bf (Weak separation)}\\
Let $C$ and $D$ be non-empty $\sigma$-stable $L^0$-convex subsets of $(L^0)^d$. Then there exists
a $Z \in (L^0)^d$ such that
\be \label{sepint}
\N{Z} > 0 \quad \mbox{and} \quad
\essinf_{X \in C} \ang{X,Z} \ge \esssup_{Y \in D} \ang{Y,Z}
\ee
if and only if $0 \notin \inn_A(C-D)$ for all $A \in {\cal F}_+$.
\end{theorem}

\begin{proof}
If there is an $A \in {\cal F}_+$ such that $0 \in \inn_A(C-D)$,
there can exist no $Z \in (L^0)^d$ such that \eqref{sepint} holds. Hence,
\eqref{sepint} implies $0 \notin \inn_A(C-D)$ for all $A \in {\cal F}_+$.

To show the converse implication, assume that $0 \notin \inn_A(C-D)$ for all $A \in {\cal F}_+$. 
Clearly, $C-D$ is $\sigma$-stable and $L^0$-convex. 
Therefore, one has $\cone(C-D) = \crl{\lambda X : \lambda \in L^0_{++}, X \in C-D}$,
from which it can be seen that $\cone(C-D)$ is $\sigma$-stable and satisfies
$1_A \cone(C-D) \not= 1_A (L^0)^d$ for all $A \in {\cal F}_+$.
So one obtains from Lemma \ref{lemmasepcone} that there exists a $Z \in (L^0)^d$ such that 
$$
\N{Z} > 0 \quad \mbox{and} \quad
\essinf_{X \in E} \ang{X,Z} \ge 0.
$$
This implies \eqref{sepint}.
\end{proof}

\begin{corollary} \label{corsepopen}
Let $C$ and $D$ be two non-empty $\sigma$-stable 
$L^0$-convex subsets of $(L^0)^d$ such that 
$1_A C$ is disjoint from $1_A D$ for all $A \in {\cal F}_+$ and $D$ is $L^0$-open.
Then there exists a $Z \in (L^0)^d$ such that
$$
\essinf_{X \in C} \ang{X,Z} > \ang{Y,Z} \quad \mbox{for all } Y \in D.
$$
\end{corollary}

\begin{proof}
It follows from Theorem \ref{thmweaksep} that there exists a $Z \in (L^0)^d$ such that
$$
\N{Z} > 0 \quad \mbox{and} \quad
\essinf_{X \in C} \ang{X,Z} \ge \esssup_{V \in D} \ang{V,Z},
$$
and since $D$ is $L^0$-open, one has
$$
\esssup_{V \in D} \ang{V,Z} > \ang{Y,Z} \quad \mbox{for all } Y \in D.
$$
\end{proof}

As another consequence of Theorem \ref{thmweaksep} we obtain a  
conditional version of the Hahn--Banach extension theorem. 

\begin{corollary} {\bf (Conditional version of the Hahn--Banach extension theorem)}\\
Let $f : (L^0)^d \to L^0$ be an $L^0$-convex function such that 
$f(\lambda X ) = \lambda f(X)$ for all $\lambda \in L^0_+$ and 
$g : E \to L^0$ an $L^0$-linear mapping on a $\sigma$-stable $L^0$-linear subset $E$ of $(L^0)^d$ 
such that $g(X ) \le f(X)$ for all $X \in E$. Then there exists an $L^0$-linear extension 
$h : (L^0)^d \to L^0$ of $g$ such that $h(X) \le f(X)$ for all $X \in (L^0)^d$.
\end{corollary}

\begin{proof}
Note that
$$
C := \crl{(X,V) \in (L^0)^d \times L^0 : f(X) \le V} \quad \mbox{and} \quad D := \crl{(Y,g(Y)) : Y \in E}
$$
are $L^0$-convex sets in $(L^0)^d \times L^0$. By Lemma \ref{lemmastablef}, 
$f$ and $g$ are stable. It follows that $C$ and $D$ are $\sigma$-stable.
Moreover, since $C-D$ is an $L^0$-convex cone and $1_A(0,-1) \notin 1_A(C-D)$
for all $A \in {\cal F}_+$, one has $(0,0) \notin \inn_A (C-D)$ for all $A \in {\cal F}_+$. 
So one obtains from Theorem 
\ref{thmweaksep} that there exists a pair $(Z,W) \in (L^0)^d \times L^0$ such that 
\be \label{HBinfsup}
\N{Z} +|W| > 0 \quad \mbox{and} \quad
\essinf_{(X,V) \in C} \crl{\ang{X,Z} + VW} \ge \esssup_{Y \in E} \crl{\ang{Y,Z} + g(Y)W}.
\ee
It follows that $W > 0$. By multiplying $(Z,W)$ with $1/W$, one can assume that $W = 1$. Since 
$E$ and $g$ are $L^0$-linear, the ess\,sup in \eqref{HBinfsup} must be zero, and it follows that
$g(Y) = \ang{Y,-Z}$ for all $Y \in E$. Moreover, 
$f(X) \ge \ang{X,-Z}$ for all $X \in (L^0)^d$. So $h(X) := \ang{X,-Z}$ is the desired extension of $g$
to $(L^0)^d$.
\end{proof}

\begin{theorem} \label{thmpropsep} {\bf (Proper separation)}\\
Let $C$ and $D$ be two non-empty $\sigma$-stable $L^0$-convex subsets of $(L^0)^d$. Then there exists
a $Z \in (L^0)^d$ such that
\be \label{sepri}
\essinf_{X \in C} \ang{X,Z} \ge \esssup_{Y \in D} \ang{Y,Z} \quad \mbox{and} \quad
\esssup_{X \in C} \ang{X,Z} > \essinf_{Y \in D} \ang{Y,Z}
\ee
if and only if
$0 \notin \ri_A(C-D)$ for all $A \in {\cal F}_+$.
\end{theorem}

\begin{proof}
Denote $E= \aff(C-D)$. By Corollary \ref{coraffsigma}, $1_A E$ is for all $A \in {\cal F}_+$ 
$\sigma$-stable, and therefore, by Proposition \ref{propaffclosed}, sequentially closed.

If there exists an $A \in {\cal F}_+$ such that
$0 \in \ri_A(C-D)$, $1_A E$ is $L^0$-linear and there exists an $\varepsilon \in L^0_{++}$ such that
$B^A_{\varepsilon}(0) \cap 1_A E \subseteq 1_A(C-D)$. Suppose there exists
$Z \in (L^0)^d$ satisfying \eqref{sepri}. Then
\be \label{properineq1}
\ang{X,Z} \ge 0 \mbox{ for all } X \in \cl_A(C-D)
\ee
and
\be \label{properineq2}
\ang{X,Z} > 0 \mbox{ on } A \mbox{ for some } X \in 1_A(C-D).
\ee
One obtains from Corollary \ref{cordec}
that $Z = Z_1 + Z_2$ for some $Z_1 \in 1_A E$ and $Z_2 \in (1_A E)^{\perp}$. 
It follows from \eqref{properineq1} that $Z_1 = 0$. But this contradicts \eqref{properineq2}.
So \eqref{sepri} implies that $0 \notin \ri_A(C-D)$ for all $A \in {\cal F}_+$.

Now assume $0 \notin \ri_A(C-D)$ for all $A \in {\cal F}_+$. 
Since $E$ is $\sigma$-stable, there exists a largest $B \in {\cal F}$ such that 
$0 \in 1_B E$. If $\mu[B] = 0$, one has $0 \notin 1_A E$ for all $A \in {\cal F}_+$, and it
follows from Corollary \ref{corsepcl} that there exists a $Z \in (L^0)^d$ such that 
$\essinf_{X \in E} \ang{X,Z} > 0$, which implies \eqref{sepri}. 
If $\mu[B] > 0$, denote $A := \Omega \setminus B$. The same argument as before yields a $Z_0 \in 1_A (L^0)^d$
satisfying \eqref{properineq1}--\eqref{properineq2}. On the other hand,
$1_B E$ is $L^0$-linear. So it follows from 
Corollary \ref{corspan} that there exist disjoint sets $B_1, \dots, B_d \in {\cal F}$ satisfying 
$\bigcup_{i=1}^d B_i = B$ and an orthonormal basis $X_1, \ldots, X_d$ of $(L^0)^d$ on 
$B$ such that $1_{B_i} E = \lin_{B_i} \{X_1,\ldots,X_i\}$ for all $i =1, \dots, d$. For every 
$i \in {\cal I} := \crl{j =1, \dots, d : \mu[B_j] > 0}$ one can apply
Theorem \ref{thmweaksep} in the $L^0$-linear subset $1_{B_i} E$ 
to obtain a $Z_i \in 1_{B_i} E$ such that 
$$
\N{Z_i} > 0 \mbox{ on } B_i \quad \mbox{and} \quad
\essinf_{X \in C} \ang{X,Z_i} \ge \esssup_{Y \in D} \ang{Y,Z_i}.
$$
Since $0 \notin \ri_A(C-D)$ for all $A \in {\cal F}_+$, one has
$$
\esssup_{X \in C} \ang{X,Z_i} > \essinf_{Y \in D} \ang{Y,Z_i} \quad \mbox{on } B_i.
$$
If one sets $Z = 1_A Z_0 + \bigcup_{i \in {\cal I}} 1_{B_i} Z_i$, one obtains \eqref{sepri}, 
and the proof is complete.
\end{proof}

\setcounter{equation}{0}
\section{Properties of $L^0$-convex functions}
\label{sec:convex}

\begin{definition} \label{defsubgrad}
Consider a function $f : (L^0)^d \to L$ and an $X_0 \in (L^0)^d$.
\begin{itemize}
\item
We call $Y \in (L^0)^d$ an $L^0$-subgradient of $f$ at $X_0$ if 
$$
f(X_0) \in L^0 \quad \mbox{and} \quad 
f(X_0 + X) - f(X_0) \ge \ang{X,Y} \quad \mbox{for all } X \in (L^0)^d.
$$
By $\partial f(X_0)$ we denote the set of all $L^0$-subgradients of $f$ at $X_0$.
\item
If $f(X_0) \in L^0$ and for some $X \in (L^0)^d$ the limit
$$
f'(X_0;X) := \lim_{n \to \infty} n \edg{f(X_0 + X/n) - f(X_0)}
$$
exists a.e. ($+\infty$ and $-\infty$ are allowed as limits), we call it $L^0$-directional derivative
of $f$ at $X_0$ in the direction $X$.
\item 
We say $f$ is $L^0$-differentiable at $X_0$ if $f(X_0) \in L^0$ and there 
exists a $Y \in (L^0)^d$ such that 
$$
\frac{f(X_0 + X_n) - f(X_0) - \ang{X_n,Y}}{\N{X_n}} \to 0 \mbox{ a.e.}
$$
for every sequence $(X_n)_{n \in \mathbb{N}}$ in $(L^0)^d$ satisfying 
$X_n \to 0$ a.e. and $\N{X_n} > 0$ for all $n \in \mathbb{N}$. If such a $Y$ exists, we call 
it the $L^0$-derivative of $f$ at $X_0$ and denote it by $\nabla f(X_0)$.
\item
The $L^0$-convex conjugate $f^* : (L^0)^d \to L$ is given by 
$$
f^*(Y) := \esssup_{X \in (L^0)^d} \crl{\ang{X,Y} - f(X)}.
$$
\item
If $f$ is $L^0$-convex, we set
$$
\dom f := \crl{X \in (L^0)^d : f(X) < + \infty}.
$$
\item
By $\conv f$ we denote the largest $L^0$-convex function below $f$
and by $\underline{\conv} f$ the largest sequentially lsc $L^0$-convex function below $f$.
\item
If $f$ is $L^0$-convex and satisfies $f(\lambda X) = \lambda f(X)$ for all 
$\lambda \in L^0_{++}$ and $X \in (L^0)^d$, we call $f$ $L^0$-sublinear.
\item 
For every pair $(Y,Z) \in (L^0)^d \times L^0$ we denote by $f^{Y,Z}$ the function 
from $(L^0)^d$ to $L^0$ given by $f^{Y,Z}(X) := \ang{X,Y} + Z$.
\end{itemize}
\end{definition}

\begin{theorem} \label{thmconvexcont}
Let  $f: (L^0)^d \to L$ be an $L^0$-convex function and $X_0 \in \inn(\dom f)$ such that $f(X_0) \in L^0$.  
Then $f(X) \in \overline{L}$ for all $X \in (L^0)^d$ and $f$ is sequentially continuous on $\inn(\dom f)$.
\end{theorem}

\begin{proof}
Since $X_0 \in \inn(\dom f)$, there exists an $\varepsilon \in L^0_{++}$ 
such that $V := \max_i f(X_0 \pm \varepsilon e_i) < + \infty$. 
By $L^0$-convexity, one has $f(X) \le V$ for all $X \in X_0 + U$, where 
$$
U := \crl{X \in (L^0)^d: \sum_{i=1}^d |X^i| \le \varepsilon}.
$$
Assume that there exist $X \in (L^0)^d$ and $A \in {\cal F}_+$ such that 
$f(X) = - \infty$ on $A$. Then one can choose a $Z \in X_0 + U$ and a
$\lambda \in L^0$ such that $0 < \lambda \le 1$ and $X_0 = \lambda X + (1-\lambda)Z$.
It follows that $f(X_0) \le \lambda f(X) + (1-\lambda) f(Z) = - \infty$ on $A$.
But this contradicts the assumptions. So $f(X) \in \overline{L}$ for all $X \in (L^0)^d$.

Now pick an $X \in U$ and a $\lambda \in L^0$ such that $0 < \lambda \le 1$. Then 
$$
f(X_0 + \lambda X) = f(\lambda (X_0 + X) + (1-\lambda) X_0) \le \lambda f(X_0+X) + (1-\lambda) f(X_0),
$$
and therefore,
$$
f(X_0 + \lambda X) - f(X_0) \le \lambda [f(X_0 + X) - f(X_0)] \le \lambda (V - f(X_0)).
$$
On the other hand, 
$$
X_0 = \frac{1}{1+ \lambda} (X_0 + \lambda X) + \frac{\lambda}{1 + \lambda} (X_0 - X).
$$
So 
$$
f(X_0) \le \frac{1}{1+ \lambda} f(X_0 + \lambda X) + \frac{\lambda}{1 + \lambda} f(X_0 - X),
$$
which gives
$$
f(X_0) - f(X_0 + \lambda X) \le \lambda [f(X_0 - X) - f(X_0)] \le \lambda (V - f(X_0)).
$$
Hence, we have shown that 
$$
|f(X) - f(X_0)| \le \lambda (V - f(X_0)) \quad \mbox{for all } X \in X_0 + \lambda U.
$$
Let $(X_n)_{n \in \mathbb{N}}$ be a sequence in $(L^0)^d$ converging a.e. to $X_0$.
For every $k \in \mathbb{N}$, the sets
$$
A^k_m := \bigcap_{n \ge m} \crl{X_n - X_0 \in U/k}
$$
are increasing in $m$ with $\bigcup_{m \ge 1} A^k_m = \Omega$.
By Lemma \ref{lemmastablef}, $f$ is stable. Therefore, 
$$
|f(X_n) - f(X_0)| \le (V - f(X_0))/k \quad \mbox{for all } n \ge m \quad \mbox{on } A^k_m,
$$
and one obtains 
$$
\mu \edg{\bigcup_{k \ge 1} \bigcap_{m \ge 1} \bigcup_{n \ge m}
\crl{|f(X_n) - f(X_0)| > (V - f(X_0))/k}} = 0.
$$
So $f(X_n) \to f(X_0)$ a.e., and the theorem follows.
\end{proof}

As an immediate consequence of Theorem \ref{thmconvexcont} one obtains the following

\begin{corollary} \label{corconvexcotinuous}
An $L^0$-convex function $f: (L^0)^d \to \overline{L}$ is sequentially continuous on $\inn (\dom f)$.
\end{corollary}

\begin{theorem} \label{thmexsubdiff}
Let $f : (L^0)^d \to \overline{L}$ be an $L^0$-convex function and $X_0 \in \ri(\dom f)$. Then 
$\partial f(X_0) \not= \emptyset$. In particular, if $f(X) \in L^0$ for all $X \in (L^0)^d$, then 
$\partial f(X_0) \not= \emptyset$ for all $X \in (L^0)^d$.
\end{theorem}

\begin{proof}
By Lemma \ref{lemmastablef}, $f$ is stable. Therefore,
$$
C := \crl{(X,V) \in (L^0)^d \times L^0 : f(X) \le V}
$$
is an $L^0$-convex, $\sigma$-stable subset of $(L^0)^d \times L^0$.
Since $(X_0, f(X_0)+ 1) $ is in $C$, one has 
$(0,0) \notin \ri_A(C- (X_0,f(X_0))$ for all $A \in {\cal F}_+$. 
So it follows from Theorem \ref{thmpropsep} that there exists 
$(Y,Z) \in (L^0)^d \times L^0$ such that 
\be \label{infv}
\essinf_{(X,V) \in C} \crl{\ang{X,Y} + VZ} \ge \ang{X_0,Y} + f(X_0)Z
\ee and
\be \label{supv}
\esssup_{(X,V) \in C} \crl{\ang{X,Y} + VZ} > \ang{X_0,Y} + f(X_0)Z.
\ee
\eqref{infv} implies that $Z \ge 0$. Now assume there exists an $A \in {\cal F}_+$ 
such that $1_A Z =0$. Then since $X_0 \in \ri (\dom f)$, \eqref{supv} contradicts \eqref{infv}. 
So one must have $Z > 0$, and by multiplying $(Y,Z)$ with $1/Z$, one can assume $Z=1$. 
It follows from \eqref{infv} that
$$
\essinf_{X \in \dom f} \crl{\ang{X,Y} + f(X)} = \ang{X_0,Y} + f(X_0),
$$
which shows that $-Y$ is an $L^0$-subgradient of $f$ at $X_0$.
\end{proof}

\begin{lemma} \label{lemmamon*}
Let $f,g : (L^0)^d \to L$ be functions such that $f \ge g$. Then the following hold:
\begin{itemize}
\item[{\rm (i)}]
$f^*$ is sequentially lsc and $L^0$-convex;
\item[{\rm (ii)}]
$f^*(Y) \ge \ang{X,Y} - f(X)$ for all $X,Y \in (L^0)^d$;
\item[{\rm (iii)}]
$Y \in \partial f(X)$ if and only if $f(X) \in L^0$ and
$f^*(Y) = \ang{X,Y} - f(X)$;
\item[{\rm (iv)}] 
$f^* \le g^*$ and $f^{**} \ge g^{**}$;
\item[{\rm (v)}]
$f \ge f^{**}$ and $f^* = f^{***}$.
\end{itemize}
\end{lemma}

\begin{proof}
To prove (i) let $(Y_n)_{n \in \mathbb{N}}$ be a sequence in $(L^0)^d$ converging a.e. to some 
$Y \in (L^0)^d$. Then 
\beas
\liminf_{n \to \infty} f^*(Y_n) 
&=& \sup_{m \ge 1} \inf_{n \ge m} \esssup_{X \in (L^0)^d} \crl{\ang{X,Y_n} - f(X)}\\
&\ge& \esssup_{X \in (L^0)^d} \sup_{m \ge 1} \inf_{n \ge m} \crl{\ang{X,Y_n} - f(X)}\\
&=& \esssup_{X \in (L^0)^d} \crl{\ang{X,Y} - f(X)} = f^*(Y).
\eeas
Hence, $f^*$ is sequentially lsc. To show that it is $L^0$-convex, choose 
$Y,Z \in (L^0)^d$ and $\lambda \in L^0$ such that $0 \le \lambda \le 1$.
Then, $\lambda f^*(Y) + (1-\lambda)f^*(Z) \ge \ang{X, \lambda Y + (1-\lambda)Z} - f(X)$
for all $X \in (L^0)^d$ and therefore, $\lambda f^*(Y) + (1-\lambda)f^*(Z) \ge f^*(\lambda Y +
(1-\lambda)Z)$.
(ii) is immediate from the definition of $f^*$. 
Now assume that $f(X) \in L^0$. For any $X' \in (L^0)^d$, 
$f(X') -f(X) \ge \ang{X'-X,Y}$ is equivalent to 
$\ang{X,Y} - f(X) \ge \ang{X',Y} - f(X')$. This shows (iii).
(iv) is clear. From (ii) one obtains that $f(X) \ge \ang{X,Y} - f^*(Y)$ for all $X,Y \in (L^0)^d$.
So $f \ge f^{**}$. The same inequality applied to $f^*$ gives 
$f^* \ge f^{***}$. On the other hand, we know from (iv) that $f^* \le f^{***}$.
This proves $(v)$.
\end{proof}

\begin{lemma} \label{lemmalscf}
Let $f : (L^0)^d \to \overline{L}$ be a
sequentially lsc $L^0$-convex function. Then one has for 
all $X \in (L^0)^d$,
$$
f(X) = \esssup \crl{f^{Y,Z}(X) : (Y,Z) \in (L^0)^d \times L^0, \, f \ge f^{Y,Z}}.
$$
\end{lemma}

\begin{proof}
Note that the set 
$$
{\cal A} := \crl{A \in {\cal F} : \mbox{ there exists an $X \in (L^0)^d$ such that 
$1_A f(X) \in L^0$}}
$$
is directed upwards. Therefore, there exists an increasing sequence 
$A_n$ in ${\cal A}$ with corresponding $X_n$, $n \in \mathbb{N}$, 
such that $A_n \uparrow A := \esssup {\cal A}$ a.e. Set
$$
X_0 := 1_{A_1 \cup A^c} X_1 + \sum_{n \ge 2} 1_{A_n \setminus A_{n-1}} X_n.
$$
By Lemma \ref{lemmastablef}, $f$ is stable. Hence,
$f(X_0) < + \infty$ on $A$, and $f(X) = + \infty$ on $A^c$ for all $X \in (L^0)^d$.
The lemma can be proved on $A$ and $A^c$ separately, and on $A^c$ it is obvious.
Therefore, we can assume $A = \Omega$.
Then $\dom f \not= \emptyset$, and it follows that
$$
C := \crl{(X,V) \in \dom f \times L^0 : f(X) \le V}
$$
is a non-empty sequentially closed $L^0$-convex subset of $(L^0)^d \times L^0$.
Choose a pair $(U,W) \in (L^0)^d \times L^0$ such that 
$1_A (U,W) \notin 1_A C$ for all $A \in {\cal F}_+$.
By Corollary \ref{corsepcl}, there exists $(Y,Z) \in (L^0)^d \times L^0$ such that
$$
I := \inf_{(X,V) \in C} \crl{\ang{X,Y} + VZ} > \ang{U,Y} + W Z.
$$
It follows that $Z \ge 0$. On the set $B := \crl{Z > 0}$ one can 
multiply $(Y,Z)$ with $1/Z$ and assume $Z = 1$. Then one obtains that on $B$,
$$
f(X) \ge f^{-Y,I}(X) \quad \mbox{for all } X \in (L^0)^d \quad \mbox{and} \quad f^{-Y,I}(U) > W.
$$
On $B^c$ one has $\lambda := I - \ang{U,Y} >0$. Pick a $U' \in \dom f$. 
Since $1_A(U', f(U') -1) \notin 1_A C$ for all $A \in {\cal F}_+$, one obtains from 
Corollary \ref{corsepcl} that there exists a pair $(Y',Z') \in (L^0)^d \times L^0$ such that
$$
I' := \inf_{(X,V) \in C} \crl{\ang{X,Y'} + VZ'} > \ang{U',Y'} + (f(U')-1) Z'.
$$
Since $U' \in \dom f$, one must have $Z' > 0$. By multiplying with $1/Z'$, one can assume $Z' = 1$.
Now choose a $\delta \in 1_{B^c} L^0_+$ such that
$$
\delta > \frac{1}{\lambda}(W + \ang{U,Y'} - I')^+ \quad \mbox{on } B^c
$$
and set $Y'' := \delta Y + Y'$. Then, on $B^c$,
\beas
I'' &:=& \inf_{(X,V) \in C} (\ang{X,Y''} + V) \ge \delta I + I' 
= \delta \lambda + \delta \ang{U,Y} + I' >  \ang{U,Y''} + W.
\eeas
So on $B^c$, one has 
$$
f(X) \ge f^{-Y'',I''}(X) \quad \mbox{for all } X \in (L^0)^d \quad \mbox{and} \quad f^{-Y'',I''}(U) > W.
$$
Now define $(\hat{Y}, \hat{I}) := 1_B(-Y,I) + 1_{B^c}(-Y'',I'')$. Then 
$$
f(X) \ge f^{\hat{Y},\hat{I}}(X) \quad \mbox{for all }X \in (L^0)^d \quad \mbox{and} \quad f^{\hat{Y},\hat{I}}(U) > W.
$$
This proves the lemma.
\end{proof}

\begin{theorem} \label{thmFM} {\bf (Conditional version of the Fenchel--Moreau theorem)}\\
Let $f : (L^0)^d \to \overline{L}$ be a function such that
$\underline{\conv} f$ takes values in $\overline{L}$. Then 
$\underline{\conv} f = f^{**}$. 
In particular, if $f$ is sequentially lsc and $L^0$-convex, then $f = f^{**}$.
\end{theorem}

\begin{proof}
We know from Lemma \ref{lemmamon*} that $f^{**}$ is a 
sequentially lsc $L^0$-convex minorant
of $f$. So $\underline{\conv} f \ge f^{**}$.
On the other hand, it follows from Lemma \ref{lemmalscf} that
$$\underline{\conv} f = \esssup \crl{f^{Y,Z}(X) : (Y,Z) \in (L^0)^d \times L^0, \, 
\underline{\conv} f \ge f^{Y,Z}},
$$
and it can easily be checked that $(f^{Y,Z})^{**} = f^{Y,Z}$ for all $(Y,Z) \in (L^0)^d \times L^0$.
So one obtains from Lemma \ref{lemmamon*} that 
$f^{**} \ge (f^{Y,Z})^{**} = f^{Y,Z}$ for every pair $(Y,Z) \in (L^0)^d \times L^0$ 
satisfying $f \ge f^{Y,Z}$. This shows that $f^{**} \ge \underline{\conv} f$.
\end{proof}

\begin{lemma} \label{lemmag}
Let $f : (L^0)^d \to L$ be an $L^0$-convex function and $X_0 \in (L^0)^d$ such that
$f(X_0) \in L^0$. Then $f'(X_0;X)$ exists for all $X \in (L^0)^d$, $f'(X_0,0) = 0$ and
$f'(X_0;.)$ is $L^0$-sublinear. Moreover, $\partial f(X_0) = \partial g(0)$, where $g(X) := f'(X_0;X)$.
\end{lemma}

\begin{proof}
It follows from $L^0$-convexity that for every $X \in (L^0)^d$, 
$n[f(X_0+ X/n) - f(X_0)]$ is decreasing in $n$. This implies that 
$f'(X_0;X)$ exists. $f'(X_0;0) = 0$ is clear. That
$f'(X_0;.)$ is $L^0$-sublinear and $\partial f(X_0) = \partial g(0)$ are straightforward to check.
\end{proof}

\begin{lemma} \label{lemmasublinear}
Let $f : (L^0)^d \to \overline{L}$ be a sequentially lsc $L^0$-sublinear function.
If there exists an $X_0 \in (L^0)^d$ such that $f(X_0) \in L^0$, then $\partial f(0) \not= \emptyset$ and
$f(X) = \esssup_{Y \in \partial f(0)} \ang{X,Y}$ for all $X \in (L^0)^d$.
In particular, $f(0) = 0$.
\end{lemma}

\begin{proof}
By Theorem \ref{thmFM}, one has $f = f^{**}$. This implies that the set 
$$
C := \crl{Y \in (L^0)^d : \ang{X,Y} \le f(X) \mbox{ for all } X \in (L^0)^d}
$$
is non-empty and $f(X) = \esssup_{Y \in C} \ang{X,Y}$.
It follows that $f(0) = 0$ and $\partial f(0) = C$. This proves the lemma.
\end{proof}

\begin{theorem}
Let $f : (L^0)^d \to \overline{L}$ be an $L^0$-convex function. Assume there exist $X_0 \in (L^0)^d$ 
and $V \in L^0_+$ such that $f(X_0) \in L^0$ and 
\be \label{boundbelow}
f(X_0 + X) \ge f(X_0) - V \N{X} \quad \mbox{for all } X \in (L^0)^d.
\ee
Then there exists a $Y \in \partial f(X_0)$ such that $\N{Y} \le V$.
\end{theorem}

\begin{proof}
Denote $g(X) := f'(X_0;X)$. Then $h = \underline{\conv}g$ is a
sequentially lsc $L^0$-sublinear function which by \eqref{boundbelow}, satisfies 
\be \label{hboundbelow}
h(X) \ge - V \N{X} \quad \mbox{for all} \quad X \in (L^0)^d.
\ee 
It follows that 
$h(0) = 0$ and $\partial h(0) \subseteq \partial g(0) = \partial f(X_0)$. Since 
$\partial h(0)$ and \[B^V(0) := \crl{Y \in (L^0)^d : \N{Y} \le V}\] are
$L^0$-convex and sequentially closed, they are both $\sigma$-stable.
Therefore, there exists a largest set $A \in {\cal F}$ 
such that $1_A \partial h(0) \cap 1_A B^V(0)$ is non-empty.
Assume that $A^c \in {\cal F}_+$. Then, if one restricts attention to $A^c$ and assumes 
$\Omega = A^c$, the sets $\partial h(0)$ and $B^V(0)$ satisfy the assumptions of Corollary \ref{corsepcl}.
So there exists a $Z \in (L^0)^d$ such that
$$
- V \N{Z} = \essinf_{Y \in B^V(0)} \ang{Y,Z} > \esssup_{Y \in \partial h(0)} \ang{Y,Z}.
$$
But by Lemma \ref{lemmasublinear}, one has $h(Z) = \esssup_{Y \in \partial h(0)} \ang{Y,Z}$,
and one obtains a contradiction to \eqref{hboundbelow}. It follows that $A = \Omega$, which proves
the theorem.
\end{proof}

\begin{theorem} \label{thmconvdiff}
Let $f : (L^0)^d \to \overline{L}$ be an $L^0$-convex function and $X_0$ in 
$(L^0)^d$ such that $f(X_0) \in L^0$. Assume that
$\partial f(X_0) = \crl{Y}$ for some $Y \in (L^0)^d$. Then 
$f$ is $L^0$-differentiable at $X_0$ with $\nabla f(X_0) = Y$.
\end{theorem}

\begin{proof}
By Lemma \ref{lemmag}, one has $\partial g(0) = \crl{Y}$ for the $L^0$-sublinear function 
$g(X) := f'(X_0;X)$. It follows that 
\be \label{g*}
g^*(Z) = 1_{\crl{Z \neq Y}} (+ \infty) \quad \mbox{and} \quad g^{**}(X) = \ang{X,Y}.
\ee Set
$$
{\cal A} := \crl{A \in {\cal F} : \mbox{ there exists an } X \in (L^0)^d \mbox{ such that }
g(X) = + \infty \mbox{ on } A}.
$$
By Lemma \ref{lemmastablef}, $g$ is stable. Therefore, there exists a sequence 
$(A_n)_{n \in \mathbb{N}}$ in ${\cal A}$ with corresponding $X_n$ such that 
$A_n \uparrow A := \esssup {\cal A}$. The element 
$$
X_0 := 1_{A_1 \cup A^c} X_1 + \sum_{n \ge 2} 1_{A_n \setminus A_{n-1}} X_n
$$
satisfies $g(X_0) = + \infty$ on $A$. We want to show that $\mu[A] = 0$. 
So let us assume $\mu[A] >0$. If one replaces $\Omega$ with $A$, one has
$0 \notin 1_B(\dom g - X_0)$ for all $B \in {\cal F}_+$. By Theorem \ref{thmweaksep},
there exists a $Z \in (L^0)^d$ such that 
$$
\N{Z} > 0 \quad \mbox{and} \quad
\essinf_{X \in \dom g} \ang{X,Z} \ge \ang{X_0,Z}. 
$$
Define the sequentially lsc $L^0$-convex function $h : (L^0)^d \to \overline{L}$ as follows:
$$
h(X) := \ang{X,Y} 1_{\crl{\ang{X,Z} \ge \ang{X_0,Z}}} + \infty 1_{\crl{\ang{X,Z} < \ang{X_0,Z}}}.
$$
Then $g \ge h$ and $h(X) = + \infty$ for all $X \in (L^0)^d$ satisfying $\ang{X,Z} < \ang{X_0,Z}$.
It follows that $\underline{\rm conv}g(X) = + \infty$ for all $X \in (L^0)^d$ satisfying $\ang{X,Z} < \ang{X_0,Z}$.
Moreover, since $Y \in \partial g(0)$, $g$ fulfills the assumptions of Theorem \ref{thmFM}, and one 
obtains $\underline{\rm conv}g = g^{**}$, contradicting \eqref{g*}. 
So one must have $\mu[A] =0$, or in other words, $g(X) \in L^0$ for all $X \in (L^0)^d$.
It follows from Theorem \ref{thmconvexcont} that $g$ is sequentially continuous, and therefore,
$g(X) = g^{**}(X) = \ang{X,Y}$ for all $X \in (L^0)^d$.

Now let $(X_n)_{n \in \mathbb{N}}$ be a sequence in $(L^0)^d$ such that 
$X_n \to 0$ a.e. and $\N{X_n} > 0$ for all $n$. Denote $\N{X_n}_1 := 
\sum_{i=1}^d |X^i_n|$ and notice that there exists a constant $c >0$ such that 
$\N{X_n}_1 \le c \N{X_n}$ for all $n$. Since $g(X) = \ang{X,Y}$, one has for all $i=1, \dots, d$,
$$
\frac{f(X_0 \pm \N{X_n}_1 e_i) - f(X_0)}{\N{X_n}_1} \to \pm Y^i \quad \mbox{a.e.}
$$
Therefore,
\beas
&& \frac{f(X_0 + X_n) - f(X_0) - \ang{X_n,Y}}{\N{X_n}} \le 
c \frac{f(X_0 + X_n) - f(X_0) - \ang{X_n,Y}}{\N{X_n}_1}\\
&\le& c \sum_{i=1}^d 
\frac{|X^i_n|}{\N{X_n}_1} \crl{\frac{f(X_0 + \N{X_n}_1 \mbox{sign}(X^i_n) e_i) - f(X_0)}{\N{X_n}_1}
- \mbox{sign}(X^i_n) Y^i } \to 0 \quad \mbox{a.e.}
\eeas
\end{proof}

\setcounter{equation}{0}
\section{Inf-convolution}
\label{sec:convolution}

\begin{definition}
We define the inf-convolution of finitely many functions $f_j : (L^0)^d \to \overline{L}$, $j = 1, \dots, n$, by 
$$
\Box_{j = 1}^n f_j(X) := \essinf_{X_1 + \dots + X_n = X} \sum_{j=1}^n f_j(X_j).
$$
\end{definition}

\begin{lemma} \label{lemmainfconvcon}
If $f_j$, $j = 1, \dots, n$, are $L^0$-convex functions from $(L^0)^d$ to $\overline{L}$, then
$\Box_{j=1}^n f_j$ is $L^0$-convex too.
\end{lemma}

\begin{proof}
Denote $f = \Box_{j=1}^n f_j$.
Choose $X,Y \in (L^0)^d$ and $V,W \in \overline{L}$ such that $f(X) \le V$ and $f(Y) \le W$.
Let $\varepsilon \in L^0_{++}$ and $\lambda \in L^0$ such that $0 \le \lambda \le 1$.
By Lemma \ref{lemmastablef}, the functions $f_j$ are stable. 
Therefore, the family $\crl{\sum_j f_j(X_j) : \sum_j X_j = X}$ is directed downwards. 
So there exist sequences $X^k_j,$ $k \in \mathbb{N}$, such that 
$\sum_j X^k_j = X$ and $\sum_j f_j(X^k_j)$ decreases to 
$f(X)$ a.e. It follows that the sets $A_k := \crl{\sum_j f_j(X^k_j) \le V + \varepsilon}$
increase to $\Omega$ as $k \to \infty$. So for every $j=1,\dots, n$, 
$$
X_j := \sum_{k \ge 1} 1_{A_k \setminus A_{k-1}} X^k_j, \quad \mbox{where } 
A_0 := \emptyset.
$$
defines an element in $(L^0)^d$ such that $\sum_{j=1}^n X_j = X$ 
and $\sum_{j=1}^n f(X_j) \le V + \varepsilon$. Analogously, there exist 
$Y_j \in (L^0)^d$, $j=1, \dots, n$, such that 
$\sum_{j=1}^n Y_j = Y$ and $\sum_{j=1}^n f(Y_j) \le W + \varepsilon$. 
Set $Z_j = \lambda X_j + (1-\lambda) Y_j$. Then 
$Z := \sum_{j=1}^n Z_j = \lambda X + (1-\lambda)Y$ and 
$$
f(Z) \le \sum_{j=1}^n f_j(Z_j) \le \sum_{j=1}^n \lambda f_j(X_j) + (1-\lambda) f(Y_j)
\le \lambda V + (1-\lambda) W + \varepsilon.
$$
It follows that $f(Z) \le \lambda f(X) + (1-\lambda) f(Y)$.
\end{proof}

\begin{lemma} \label{lemmaconvint}
Let $f_j : (L^0)^d \to \overline{L}$, $j=1, \dots, n$, be $L^0$-convex functions
and denote $f = \Box_{j=1}^n f_j$. Assume $f(X_0) = \sum_{j=1}^n f_j(X_j) < + \infty$ for some 
$X_j \in (L^0)^d$ summing up to $X_0$. If $X_1 \in \inn (\dom f_1)$, then 
$f(X) \in \overline{L}$ for all $X \in (L^0)^d$, $X_0 \in \inn (\dom f)$ and 
$f$ is sequentially continuous on $\inn (\dom f)$.
\end{lemma}

\begin{proof}
By definition of $f$, one has
$$
f(X_0 + X) - f(X_0) \le  f_1(X_1 + X) + \sum_{j=2}^n f_j(X_j) - \sum_{j=1}^n f_j(X_j) =
f_1(X_1 + X) - f_1(X_1) 
$$
for all $X \in (L^0)^d$. This shows that $X_0 \in \inn(\dom f)$. Since 
$f(X_0) = \sum_{j=1}^n f_j(X_j) \in L^0$, the rest of the lemma follows from Theorem \ref{thmconvexcont}.
\end{proof}

\begin{lemma} \label{lemmaconvsub}
Consider functions $f_j : (L^0)^d \to \overline{L}$, $j=1, \dots, n$,
and denote $f = \Box_{j=1}^n f_j$. Assume $f(X_0) = \sum_{j=1}^n f_j(X_j) < + \infty$ for some 
$X_j \in (L^0)^d$ summing up to $X_0$.
Then $\partial f(X_0) = \bigcap_{j=1}^n \partial f_j(X_j)$.
\end{lemma}

\begin{proof}
Assume $Y \in \partial f(X_0)$ and $X \in (L^0)^d$. Then 
$$
f_1(X_1 + X) - f_1(X_1) = f_1(X_1 + X) + \sum_{j=2}^n f_j(X_j) - \sum_{j=1}^n f_j(X_j) 
\ge f(X_0 + X) - f(X_0) \ge \ang{X,Y}.
$$
Hence $Y \in \partial f_1(X_1)$, and by symmetry, $\partial f(X_0)
\subseteq \bigcap_{j=1}^n \partial f_j(X_j).$ On the other hand, if $Y \in \bigcap_{j=1}^n \partial f_j(X_j)$
and $X \in (L^0)^d$, choose $Z_j$ such that $\sum_{j=1}^n Z_j = X_0 + X$.
Then 
$$
\sum_{j=1}^n f_j(Z_j) \ge \sum_{j=1}^n f_j(X_j) + \ang{Z_j- X_j, Y} =  \sum_{j=1}^n f_j(X_j) + \ang{X,Y}.
$$
So $f(X_0+X) - f(X_0) \ge \ang{X,Y}$, and the lemma follows.
\end{proof}

\begin{lemma} \label{lemmaconvdiff}
Let $f_j : (L^0)^d \to \overline{L}$, $j=1, \dots, n$, be $L^0$-convex functions and denote 
$f = \Box_{j=1}^n f_j$. Assume $f(X_0) = \sum_j f_j(X_j) < +\infty$ for some 
$X_j \in (L^0)^d$ summing up to $X_0$ and $f_1$ is $L^0$-differentiable at $X_1$. 
Then $f$ is $L^0$-differentiable at $X_0$ with $\nabla f(X_0) = \nabla f_1(X_1)$.
\end{lemma}

\begin{proof}
One has
$$
f(X_0 + X) - f(X_0) \le f_1(X_1 + X) + \sum_{j=2}^n f_j(X_j) - \sum_{j=1}^n f_j(X_j) = f_1(X_1 + X) - f_1(X_1)
$$
for all $X \in (L^0)^d$. It follows that the $L^0$-directional derivative $g(X) := f'(X_0;X)$ satisfies
$$
g(X) \le f_1'(X_1;X) = \ang{X,\nabla f_1(X_1)}
$$
for all $X \in (L^0)^d$. But by Lemma \ref{lemmainfconvcon}, $f$ is $L^0$-convex. It follows that
$g$ is $L^0$-sublinear, and therefore, $g(X) = \ang{X, \nabla f_1(X_1)}$.
This implies that $\partial f(X_0) = \partial g(0) = \crl{\nabla f_1(X_1)}$. Now the lemma 
follows from Theorem \ref{thmconvdiff}.
\end{proof}

\begin{lemma} \label{lemmaconvconj}
Consider functions $f_j : (L^0)^d \to \overline{L}$, $j=1, \dots, n$.
Then $\brak{\Box_{j=1}^n f_j}^* = \sum_{j=1}^n f^*_j,$ where the sum 
is understood to be $- \infty$ if at least one of the terms is $- \infty$.
\end{lemma}

\begin{proof}
$$\brak{\Box_{j=1}^n f_j}^*(Y) = \esssup_X \crl{\ang{X,Y} - \Box_{j=1}^n f_j(X)}
= \esssup_{X_1, \dots, X_n} \sum_{j=1}^n \crl{\ang{X_j,Y} - f_j(X_j)} = \sum_{j=1}^n f^*_j(Y).
$$
\end{proof}


\begin{thebibliography}{20}

\bibitem{CH} Cheridito, P., Hu, Y. (2011).
Optimal consumption and investment in incomplete markets with general constraints.
Stochastics and Dynamics 11(2), 283--299.

\bibitem{CHKP} Cheridito, P., Horst, U., Kupper, M., Pirvu, T. (2014).
Equilibrium pricing in incomplete markets under translation invariant preferences.
SSRN Preprint.

\bibitem{CS} Cheridito, P., Stadje, M. (2012).
BS$\Delta$Es and BSDEs with non-Lipschitz drivers: comparison, convergence and robustness.
Forthcoming in Bernoulli.

\bibitem{FKV} Filipovi\'c, D., Kupper, M., Vogelpoth, N. (2009). Separation and Duality in Locally $L^0$-Convex Modules.
Journal of Functional Analysis 256, 3996--4029. 

\bibitem{FS3} F\"ollmer, H., Schied, A. (2004). Stochastic Finance,
An Introduction in Discrete Time. 2nd Edition. de Gruyter Studies in Mathematics 27.

\bibitem{G} Guo, T. (2010).
Relations between some basic results derived from two kinds of topologies for a random locally convex module.
Journal of Functional Analysis 258, 3024--3047. 

\bibitem{G2} Guo, T. (2011). Recent progress in random metric theory and its applications to conditional risk measures.
Science China Mathematics 54(4), 633--660. 



\bibitem{KS} Kabanov, Y., Stricker, Ch. (2001)
A teacher's note on no-arbitrage criteria. S\'eminaire de probabilit\'es, 35, 149--152. 

\bibitem{KV} Kupper, M., Vogelpoth, N. (2009). Complete $L^0$-normed modules and automatic continuity of monotone convex functions. Preprint.  

\bibitem{N} Neveu, J. (1975). Discrete-Parameter Martingales. North-Holland Publishing Company
Amsterdam, Oxford.


\end{thebibliography}
\end{document}